\documentclass[preprint, 1p, review]{elsarticle}
\usepackage{algorithm, algorithmicx}
\usepackage{algpseudocode} 
\usepackage{amsmath,amssymb}
\usepackage{amsfonts}
\usepackage{graphics,graphicx}
\usepackage{booktabs}	 
\usepackage{multirow}
\usepackage{float}
\usepackage{caption, subcaption}
\usepackage{lineno}

\newtheorem{theorem}{Theorem}

\newtheorem{corollary}[theorem]{Corollary}

\newtheorem{definition}[theorem]{Definition}

\newtheorem{lemma}[theorem]{Lemma}

\newtheorem{proposition}[theorem]{Proposition}
\newtheorem{remark}[theorem]{Remark}

\newenvironment{proof}[1][Proof]{\noindent\textbf{#1.} }{\ \rule{0.5em}{0.5em}}
\newcommand{\braces}[1]{\left\{ #1 \right \}}

\newcommand{\bfemph}[1]{\textbf{\emph{#1}}}

\DeclareMathOperator*{\argmin}{argmin}

\makeatletter
\def\BState{\State\hskip-\ALG@thistlm}
\makeatother

\begin{document}

\begin{frontmatter}

\title{Combining Penalty-based and Gauss-Seidel Methods for solving Stochastic Mixed-Integer Problems}

\author{F. Oliveira}	
\author{J. Christiansen}
\author{B. Dandurand}
\author{A. Eberhard}

\address{Mathematical Sciences \\ School of Science - RMIT University}
\cortext[cor]{Corresponding Author}

\begin{abstract}
In this paper, we propose a novel decomposition approach for mixed-integer stochastic programming (SMIP) problems that is inspired by the combination of penalty-based Lagrangian and block Gauss-Seidel methods (PBGS). In this sense, PBGS is developed such that the inherent decomposable structure that SMIPs present can be exploited in a computationally efficient manner. The performance of the proposed method is compared with the Progressive Hedging method (PH), which also can be viewed as a Lagrangian-based method for obtaining solutions for SMIP. Numerical experiments performed using instances from the literature illustrate the efficiency of the proposed method in terms of computational performance and solution quality. 
\end{abstract}

\journal{International Transactions of Operational Research}


\end{frontmatter}


\section{Introduction}

Inspired by recent advances and the increased availability of parallel computation resources, there has been a recent surge of methods which are capable of exploiting the structure of large-scale mathematical programming problems to achieve increased efficiency. 

One relevant class of problems that can benefit from this paradigm is stochastic mixed-integer programming (SMIP) problems. 
The modelling framework for this class of problems is versatile as it simultaneously allows for representation of integer-valued decisions and uncertainty in the input data.
However, they are frequently challenging in terms of computational tractability due to their inherent NP-hard nature and their large-scale that arise from their scenario-based representations. Parallel computation is particularly appropriate for solving SMIPs, since the special structure of these problems (i.e. their partial separation by decision stage and by outcome scenario) makes them easier to decompose into smaller subproblems which may then be solved simultaneously.

The opportunities arising from approaching SMIP problems by means of decomposition has prompted the development of several different theoretical and algorithmic approaches. For example, the Integer L-Shaped method \cite{LaporteLSH_IP1993} employs Benders' decomposition to achieve stage-wise decomposition of SMIP problems. Other algorithms employ Lagrangian duality to achieve scenario-wise decomposition, such as the Dual Decomposition algorithm \cite{CaroeDuDe1999} which uses Lagrangian dual bounds in a branch-and-bound framework, or Progressive Hedging (PH) \cite{RockafellarPH1991,LokkeTabuPH1996, WatsonWoodruff2011, Velizetal2015} which applies an Alternating-Direction-type method to the augmented Lagrangian dual problem.
Recent studies and applications of these methods include \cite{AnguloLSH_IP2016,GuoPHDD2015,LubinParDD2013,GadeLBPH2016} and references therein.

All of the above methods are based on the concept of duality, and therefore they must consider the \emph{duality gap} that may exist between the optimal solution values of the original (primal) problem and the dual problem. This duality gap is frequently nonzero in the context of non-convex problems, such as those with integer decision variables. If the duality gap for a particular problem is large, any algorithm based on that dual is unlikely to be effective \cite{ChenLD2010}.

Several possible approaches to modifying Lagrangian duality to deal with the duality gap which arises in non-convex problems have been considered in the literature. These approaches include $l_1$-like penalty functions \cite{ChenLD2010}, indicator augmenting functions \cite{LalithaIndicatorLD2010}, nonlinear Lagrangian functions \cite{YangNLLD2001}, and semi-Lagrangian duality \cite{BeltranSemiL2006}. With the exception of nonlinear Lagrangian functions, these approaches have not yet been widely exploited in terms of experimental investigation and practical applications despite numerous theoretical developments available in literature.

In this paper, we propose an alternative approach to deal with SMIP problems that builds upon recent theoretical results from \cite{BolandEberhardetal2015} and \cite{Feizollahi2016} showing that duality gaps can be diminished with the use of finite-valued penalties for specific class of penalty functions. We show that one can obtain reasonable penalty functions using \emph{positive bases} and that parallelisation can be obtained by the application of a block Gauss-Seidel approach. The combination of these two frameworks allows us to develop an efficient heuristics that is capable of providing solutions for large-scale SMIP problems. In terms of objective value quality and computational time, the developed approach is shown to be competitive with existing approaches such as PH, which, despite its heuristic nature in the context of SMIPs, has been relied upon as an efficient solution method  (see, for example \cite{RyanEtAl2013, PerboliEtAl2015, Velizetal2015}). Furthermore, the theoretical basis for PH does not apply for problems containing integer variables. On the other hand there exists some supporting theory for the PBGS approach which we present in this paper. This partial theory has the potential to inform directions for the further improvement of heuristic methods of this kind. 

This paper is structured as follows. In Section \ref{sec2} we cover the technical background which will be used in the development of the proposed method. Section \ref{sec3} describes the development of the penalty-based block Gauss-Seidel method, while in Section \ref{sec4} we discuss computational aspects of the algorithm. Section \ref{sec5} provides results of the numerical experiments performed. Finally, in Section \ref{sec6} we provide conclusions and directions for further development of this research.

\section{Technical Background} \label{sec2}

In the following developments, we consider two-stage stochastic mixed-integer programming problems of the form:
\begin{align}
\zeta^{SIP}:= ~ \min_{x,y} \ &c^\top x + \sum_{s \in S}p_s(q_s^\top y_s) \label{OriginalDE_start}\\
\text{s.t.: } &x \in X \\
&y_s \in Y_s(x), \ \forall s \in S, \label{OriginalDE_end} 
\end{align}
where $x \in \mathbb{R}^{n_x}, y \in \mathbb{R}^{n_y\times|S|}$ are decision variables, $c \in \mathbb{R}^{n_x}$ and $q \in \mathbb{R}^{n_y \times |S|}$ are input parameters, and $p_s \in \mathbb{R}^{|S|}$ represents the scenario probabilities. Sets $X \subset \mathbb{R}^{n_x}$ and $Y_s(x) \subset \mathbb{R}^{n_y \times |S|}$ define the feasible decision set and consist of linear constraints and integrality restrictions on $x$ and $y$.

To obtain a formulation that is amenable to decomposition, i.e., that can be exploited in terms of its block-angular structure, $\zeta^{SIP}$ may be equivalently rewritten as:
\begin{flalign}
\zeta^{SIP} := ~ \min_{x,y,z} \ &\sum_{s \in S}p_s(c^\top x_s + q_s^\top y_s) \label{SplitVarStart}\\
\text{s.t.: } 
&x_s - z = 0, \ \forall s \in S \label{NAC} \\
&x_s \in X, \ \forall s \in S  \\
&y_s \in Y_s(x_s), \ \forall s \in S. \label{SplitVarEnd}
\end{flalign}
The set of constraints represented by \eqref{NAC} are referred to in the relevant literature as the non-anticipativity constraints (NAC), as they enforce consensus over the decision made before the observation of a given scenario $s \in S$. It is straightforward to see that these constraints prevent the problem from being solved by means of a decomposition approach that can exploit the otherwise block-angular structure of the problem. Hence, a natural way to approach this class of problems consists of relying on frameworks that are capable of considering relaxations for $\zeta^{SIP}$ which does not include constraint \eqref{NAC}.

\subsection{Lagrangian Relaxation}

A natural approach to solve this problem is to relax the NAC by means of Lagrangian relaxation. To achieve this, let $\lambda = (\lambda_s)_{s \in S}$ be the Lagrangian multipliers associated with the constraints \eqref{NAC}. Then, our Lagrangian relaxation can be formulated as the following problem:
\begin{flalign*}
\zeta^{LR}(\omega) := ~ \min_{x,y,z} \ &\sum_{s \in S}p_s L_s(x_s,y_s,z,\omega) \\
\text{s.t.: } &x_s \in X, \ \forall s \in S  \\
&y_s \in Y_s(x_s), \ \forall s \in S,
\end{flalign*}
where $\omega := (\omega_s)_{s \in S} = \left(\frac{\lambda_s}{p_s}\right)_{s \in S}$ and 
\begin{equation}
L_s(x_s,y_s,z,\omega) := c^\top x_s + q_s ^\top y_s + \omega_s^\top(x_s-z). \label{originalLag}
\end{equation}
In order to guarantee that $\zeta^{LR}(\omega)$ has a bounded optimal solution, one must enforce that the dual feasibility condition $\omega \in \Omega := \{\omega  \mid \sum_{s \in S} p_s^\top \omega_s = 0\}$ holds. Under this assumption, \eqref{originalLag} may be rewritten as 
\begin{equation}
L_s(x_s,y_s,z,\omega) = (c + \omega_s)^\top x_s + q_s ^\top y_s,
\end{equation}
which forces the $z$ variable to vanish (note that the term $-\omega_s^\top z$ was in fact the potential cause of the unboundedness of the dual relaxation $\zeta^{LR}(\omega)$) and its removal yields complete separability for each $s \in S$. 

It is well known that $\zeta^{LR}(\omega) \leq \zeta^{SIP}$ for any $\omega \in \Omega$. The \emph{Lagrangian dual problem} consists of finding the $\omega$ which causes $\zeta^{LR}(\omega)$ to most closely approximate or bound $\zeta^{SIP}$ from below, which in practice means solving the problem  
\begin{equation}
\zeta^{LD}  := ~ \max_{\omega \in \Omega} \ \zeta^{LR}(\omega).
\end{equation}
In general only the weak duality condition $\zeta^{LR}(\omega) \le \zeta^{SIP}$ holds. When equality $\zeta^{LD} = \zeta^{SIP}$ holds, we have strong duality.
Due to the presence of integer restricted variables, the primal problem~\eqref{SplitVarStart}-\eqref{SplitVarEnd} is not convex, and strong duality (that is $\zeta^{LD} = \zeta^{SIP}$) cannot be guaranteed. Instead, we typically have a duality gap i.e. $\zeta^{LD}  < \zeta^{SIP}$. 

\subsection{Augmented Lagrangian Approach}

In the particular domain of mixed-integer problems such as SMIP problems, there has been renewed interest in the use of augmented Lagrangian approaches \cite{BurachikEtAl2015, FeizollahiEtAl2015, GeisslerEtAl2015}. The augmented Lagrangian relaxation of $\zeta^{SIP}$ which relaxes the NAC \eqref{NAC} is:
\begin{flalign}
\zeta^{LR+}_\rho (\omega) := ~ \min_{x,y,z} \ &\sum_{s \in S}p_s L^s(x_s,y_s,z,\omega) + \psi_{\rho}^s(x_s-z) \label{ALDualFunctionStart}\\
\text{s.t.: } &x_s \in X, \ \forall s \in S   \\
&y_s \in Y_s(x_s), \ \forall s \in S, \label{ALDualFunctionEnd}
\end{flalign}
where $\omega = (\omega_s)_{s \in S} \in \Omega$ and $\psi_{\rho}^s : \mathbb{R}^{n_x} \mapsto \mathbb{R}$ is an appropriate penalty function specific to scenario $s$ that depends on the penalty parameter $\rho$.
As in the ordinary Lagrangian relaxation, $\omega \in \Omega$ implies that $\sum_{s \in S} p_s^\top \omega_s = 0$ so as to ensure that the dual problem has a finite optimal value. The \emph{augmented Lagrangian dual problem} is:
\begin{align*}
\zeta^{LD+}_\rho := \max_{\omega \in \Omega} \ &\zeta^{LR+}_\rho(\omega). 
\end{align*}

A common choice for  the penalty function, in this context, is $\psi_\rho^s(u_s) := \frac{\rho}{2} ||u_s||_2^2$ for each $s \in S$, which provides smoothness to the original scenario-wise augmented Lagrangian dual function~\cite{Rockafellar1976MO,Bertsekas1982,Bertsekas1999}.  

Recent results have shown that the augmented Lagrangian dual is capable of asymptotically achieving zero duality gap when the weight $\rho$ associated with the penalty function is allowed to go to infinity \cite[Prop. 3]{BolandEberhardetal2015}, \cite[Prop. 2 ]{Feizollahi2016}. However, despite the theoretical relevance of this observation, it is not practically meaningful to deal with large-valued penalty parameters, in large part due to the associated numerical issues that arise. 

Furthermore, \cite[Cor. 1]{BolandEberhardetal2015}, \cite[Thm. 4 ]{Feizollahi2016} demonstrates that it is possible to circumvent this drawback if the augmentation of the Lagrangian dual is made using a norm as the penalty function. In this case, the theory suggests that it is possible to attain strong duality for a finite value of $\rho$. This result is one of the major motivations for the developments to be presented next.

\subsection{Semi-Lagrangian Duality}

Semi-Lagrangian duality \cite{BeltranSemiL2006} is a variant of Lagrangian duality in which "difficult" equality constraints (e.g. $Ax = b$) are reformulated as pairs of inequality constraints ($Ax \leq b$ and $Ax \geq b$). Lagrangian relaxation is then applied to one of the two sets of inequality constraints. Surrogate semi-Lagrangian duality \cite{MonabbatiSemiL2014,JornstenSemiL2015} is a variant which replaces one set of inequalities with its weighted sum ($\lambda^\top Ax \leq \lambda^\top b$, where $\lambda \geq 0$ is a non-negative vector of the weights applied to each inequality) and then applies Lagrangian relaxation to the resulting single inequality. Special classes of problems can exhibit zero duality gap when utilising the semi-Lagrangian dual problem.

The semi-Lagrangian approach is effective when the semi-Lagrangian dual problem is more tractable than the original problem, even though some inequalities remain in the dual problem as explicit constraints. For problems to which semi-Lagrangian relaxation has been previously applied, such as the  p-median problem \cite{BeltranSemiL2006} and the uncapacitated facility location problem  \cite{BeltranRoyoSL_UCAP2012}, the semi-Lagrangian dual problem may be simplified by choosing appropriate dual variable values. 

The method presented in this paper is similar to semi-Lagrangian duality in that it can be interpreted as a penalty-method analogue of applying Lagrangian duality to a reformulation of the original problem in which equalities are rewritten as inequalities. To attain this objective, let us first reformulate the problem \eqref{OriginalDE_start}-\eqref{OriginalDE_end} into the following equivalent form:
\begin{flalign}
\zeta^{SIP} : ~ \min_{x,y,z} \ &\sum_{s \in S}p_s(c^\top x_s + q_s^\top y_s) \nonumber\\
\text{s.t.: } &x_s - z \leq 0, \ \forall s \in S \label{leftNAC} \\
&-(x_s - z)  \leq 0, \ \forall s \in S \label{rightNAC}\\
&x_s \in X, \ \forall s \in S  \nonumber \\
&y_s \in Y_s(x_s), \ \forall s \in S. \nonumber
\end{flalign}
Unfortunately, in the context of the SMIP problems studied in this paper, the inequalities left as explicit constraints by semi-Lagrangian duality would not allow us to achieve our goal of scenario-wise decomposition. 

The method presented here instead relaxes inequalities \eqref{leftNAC} and \eqref{rightNAC}, 
so that penalty-based approaches can treat the deviations from the separate inequalities differently.
Furthermore, the resulting dual problem is separable by scenario within the block Gauss-Seidel framework. As will be observed later, choosing penalty functions based on some positive bases can result in penalty terms analogous to the objective terms obtained through surrogate semi-Lagrangian duality.

\subsection{Desirable Properties of Penalty Functions}

Our primary objective is to compute $\zeta^{LD+}_\rho$ in a decomposed manner, which will require: $i)$ the definition of a suitable penalty functions $\psi_\rho^s$ for each $s \in S$; and
$ii)$ the application of a block Gauss-Seidel (GS)-based approach based on a decomposable structure.

One important result, originally proven for a general mixed-integer programming (MIP) problems, which can be used in this case is Theorem 5 of \cite{Feizollahi2016}, which is reproduced below (adapted to the context of SMIP problems).

\begin{theorem}\cite[Thm. 5 ]{Feizollahi2016} \label{Thm1}
Consider a feasible MIP problem given in \eqref{OriginalDE_start}-\eqref{OriginalDE_end} whose problem data is formed from rational entries and with its optimal value bounded. 
If $\psi : \prod_{s \in S} \mathbb{R}^{n_x} \mapsto \mathbb{R}$ is a summed augmenting function $\psi(u) := \sum_{s \in S} \psi_\rho^s(u_s)$ for problem~\eqref{ALDualFunctionStart}--\eqref{ALDualFunctionEnd} such that 
\begin{enumerate}
\item $\psi(0) = 0$
\item $\psi(u) \geq \delta > 0, \forall u \not\in V$
\item $\psi(u) \geq \gamma ||u||_\infty, \forall u \in V$
\end{enumerate}
for some open neighbourhood $V$ of $0$, and positive scalars $\delta, \gamma > 0$,
then there exists a finite $\rho$ such that $\zeta^{LD+}_\rho = \zeta^{LR+}_{\rho} (\omega_{LP}) = \zeta^{SIP}$, for $\omega_{LP}$ (an optimal multiplier of the linear programming relaxation of the NACs (\ref{NAC})). 
\end{theorem}

\begin{proof}
Apply the general theorem \cite[Thm. 5 ]{Feizollahi2016} to our problem  \eqref{OriginalDE_start}-\eqref{OriginalDE_end}. 
\end{proof}

\begin{remark}\label{BolandEberhardRem2}
	In \cite{BolandEberhardetal2015}  other conditions that do not require the assumption of rationality of the data defining the problem are given that also ensure a limiting zero duality gap. 
\end{remark}
\begin{remark}\label{BolandEberhard}
One may see with little difficulty that the proof of \cite{Feizollahi2016} does not rely on the setting of $\omega = \omega_{LP}$. Indeed one can show that for any $\omega$ there still exists a finite (possibly larger) penalty parameter such that Theorem \ref{Thm1} holds true. 
\end{remark}

In \cite{BolandEberhardetal2015} and later in \cite{Feizollahi2016} it has been observed that a zero duality gap is achievable for dual problems based on an augmented Lagrangian in MIP problems. In both papers, very general classes of augmenting functions were studied and consequently very little can be inferred as to what would be a practical penalty that one could use on a given problem. It is observed in \cite{Feizollahi2016} that the usual quadratic (squared norm) penalty is probably not a practical choice for MIP. One would hope that an augmenting function would lead to a reformulation of the MIP that is not significantly worse to solve than the original problem, which would mean that augmenting functions should lead to a MIP reformulation. 

Motivated by the aforementioned facts, we propose a class of augmenting functions based on the use of {\em positive basis} \cite{Davis1954}. 
One special case of this class of penalty functions is given as follows.
Given discrepancy vector $u:=(u_s)_{s \in S} \in \prod_{s \in S} \mathbb{R}^{n_x}$, we define for each scenario $s$ the penalty function 
$$\psi_{\rho}^s(u_s):=\underline{\rho}_s^\top[u_s]^- +  \overline{\rho}_s^\top [-u_s]^-,$$
where $\rho = (\underline{\rho}_s, \overline{\rho}_s)_{s \in S} \in \mathbb{R}^{2n_x |S|}_{>0}$ and $[ v ]^- :=  -\min\{0,  v \} $ (performed component wise), where in this case $v \in \mathbb{R}^{n_x}$.
Then we define
\begin{equation}
\psi_\rho(u):=\sum_{s \in S} \psi_{\rho}^s(u_s) = \left(\sum_{s \in S}\underline{\rho}_s^\top[u_s]^- +  \sum_{s \in S}\overline{\rho}_s^\top [-u_s]^-\right). \label{neqn:1}
\end{equation}
In the following developments we will demonstrate that (\ref{neqn:1}) satisfies the conditions of Theorem \ref{Thm1} and indeed lies in a special class of augmenting functions that form a practical set from which one can tailor make an augmenting function for a given problem. 

\subsection{Positive Basis}

A subset of reasonable augmenting functions may be defined by using a \emph{positive basis} $\{\text{\bf n}_1, \dots, \text{\bf n}_l\}$ to scalarise the deviations $u\in \mathbb{R}^{m}$. (Note that for our purposes, $m=n_x \,| S | $.)
Such deviations can be associated, for example, with the satisfaction of linear inequalities, where we might have $u=b-Ax$ given a constraint $Ax \le b$.

\begin{definition}
	We say a set of vectors $\left\{ \text{\bf n}_{1},\dots ,\ \text{\bf n}_{l}\right\}$ where ${m} +1\leq l \leq 2{m}  $ is a \emph{positive basis} for $\mathbb{R}^{m}$ if and only if every $u \in \mathbb{R}^{{m}}$ can be expressed as a positive combination of these vectors, i.e., there exists $\alpha_{i}\geq 0$ for $i=1,\dots ,l$ for which $u=\sum_{i=1}^{l}\alpha _{i}\bfemph{n}_{i}.$
\end{definition}

The following property of positive bases will be useful in the developments which follow.

\begin{theorem}\label{theorem:positive-dot-product} (\cite{Davis1954} Theorem 3.1) $\left\{\bfemph{n}_{1},\dots , \bfemph{n}_{l}\right\}$ positively spans $\mathbb{R}^{m}$ if and only if for every non-zero $u$ there exists an index $i$ such that $u\cdot \bfemph{n}_{i} > 0$.
\end{theorem}

Let $e_i$, $i=1,\dots,m$ represent the elementary unit vectors of $\mathbb{R}^{m}$ with entry $i$ set to one and all other entries set to zero. Examples of positive bases on $\mathbb{R}^{m}$ include:
\begin{itemize}
	\item The vertices of a ${m}$-simplex (generalised tetrahedron), centred at the origin.
	\item The set of vectors $\{+e_i\}^{m}_{i=1} \cup \{ \sum_{i=1}^{m} - e_i\}$
	\item The set of vectors $\{\pm e_i\}^{m}_{i=1}$
\end{itemize}

\subsection{Norm-Like Augmenting Functions}

As noted in both \cite{BolandEberhardetal2015} and \cite{Feizollahi2016}, norms are viable augmenting functions with appealing theoretical support for overcoming duality gaps. However, norms have the disadvantage that they uniformly penalise constraint violations, which 
results in a loss of flexibility and precision in the fine-tuning of the penalisation.
This limitation motivates the following development of asymmetrical but norm-like augmenting functions. The polyhedral norms $\| \ {\cdot} \ \|_\infty$ and $ \| \ {\cdot} \ \|_1$ may be represented using the positive basis $\{\pm e_i\}^{m}_{i=1}$ in the following ways:
\begin{eqnarray}
\psi_\infty(u) := \|{u}\|_\infty &=& \max_{i=1,\dots,{m}} \{ \pm e_i^\top u\}, \text{ and} \label{eq:pbnorminf} \\ 
\psi_1(u) := \| u \|_1 &=& \sum_{i=1}^{m}  \max \{ +e_i^\top u, 0 \} + \sum_{i=1}^{m} \max \{ - e_i^\top u, 0 \}  \label{eq:pbnormone} \\ 
\text{ or equivalently } \| u\|_1 &=&  \sum_{i=1}^{m} \max \left\{ \nu_i^\top u : \nu_i \in \{+e_i, -e_i \} \right\}. \label{eq:pbnormonealt}
\end{eqnarray}
The representation in \eqref{eq:pbnormonealt} relies on each vector in the basis having a negative multiple which is also in the basis. The representations in \eqref{eq:pbnorminf} and \eqref{eq:pbnormone} do not have this limitation, and may be generalised to any positive basis $N := \{\text{\bf n}_1, \dots, \text{\bf n}_l \}$ as follows:

\begin{eqnarray}\label{eq:penFinf}
\psi_{\infty}^{N}(u) &:=& \max_{i=1,\dots,l} \{ \text{\bf n}_i^\top u \} \quad \text{ and}  \label{neqn:2}\\
\label{eq:penFone}
\psi_{1}^{N}(u) &:=& \sum_{i=1}^{l} \max \{ \text{\bf n}_i^\top u, 0 \}. \label{neqn:3}
\end{eqnarray}
The functions $\psi_{\infty}^{N}$ and $\psi_{1}^{N}$ are not necessarily norms, but do share some useful properties with norms. Specifically, these functions are positive homogeneous (which implies that they vanish at zero), strictly positive for all $u \in \mathbb{R}^{m} \setminus \{0\}$, finite valued, sub-additive and coercive.

The proposed augmenting function $\psi_\rho(u)$ given in (\ref{neqn:1}) may be represented in the form of (\ref{neqn:3}), using the positive basis $N_\rho = \{\overline{\rho}_{s,i} e_{i+(s-1)n_x} \mid s \in S, i \in \{1,\dots,{n_x}\} \} \cup \{- \underline{\rho}_{s,i} e_{i+(s-1)n_x} \mid s \in S, i \in \{1,\dots,{n_x}\} \}$, as follows:



\begin{eqnarray}
\psi_{1}^{N_\rho}(u) &=& \sum_{s \in S} \sum_{i=1,\dots,{n_x}} \overline{\rho}_{s,i}  \max\{0, u_{s,i}\} + \sum_{s \in S}\sum_{i=1,\dots,{n_x}}  \underline{\rho}_{s,i}  \max\{0, -u_{s,i}\} \nonumber \\
&=& \left(\sum_{s \in S}\underline{\rho}_s^\top[u_s]^- +  \sum_{s \in S}\overline{\rho}_s^\top [ - u_s]^-\right) \label{equation:phi-mu-represented-as-phi-1}\\
&=& \sum_{s \in S} \psi_{\rho}^s(u) = \psi_\rho(u) \nonumber
\end{eqnarray}


\begin{lemma}\label{lemma:positive-homogeneous-equivalence}If two functions $\psi_A$ and $\psi_B$ are positive homogeneous, strictly positive for all $u \neq 0$, and are finite valued then there exists a finite $\gamma > 0$ such that
	\begin{equation}\label{eq:positive-homogeneous-equivalence}
	\psi_A(u) \geq \gamma \psi_B(u) \quad \text{ for all  } u\in \mathbb{R}^{m}
	\end{equation}
\end{lemma}

\begin{proof}
Since they are positive homogeneous, $\psi_A$ and $\psi_B$ vanish at zero and so \eqref{eq:positive-homogeneous-equivalence} trivially holds with equality at $u=0$.
To obtain the required inequality in \eqref{eq:positive-homogeneous-equivalence} for nonzero $u$, set $V = \{ u : ||u|| = 1\}$ (where $||\cdot||$ is any norm) and take $\gamma = \min_{u \in V} \frac{\psi_A(u)}{\psi_B(u)}$. Since $\psi_A(u)$ and $\psi_B(u)$ are strictly positive and finite for all $u \neq 0$, and $V$ is compact, this minimum exists and $\gamma$ is strictly positive and finite. For any point $u \in \mathbb{R}^{m} \setminus \{0\}$, $||u||$ is strictly positive and the point $\frac{u}{||u||}$ is in $V$. The required inequality follows from the positive homogeneity of $\psi_A$ and $\psi_B$:
\begin{multline*}
\psi_A(u) = {||u||} \psi_A\left(\frac{u}{{||u||}}\right) = {||u||} \frac{\psi_A\left(\frac{u}{||u||} \right)}{\psi_B\left(\frac{u}{||u||} \right)} \psi_B\left(\frac{u}{||u||} \right) \geq \gamma{||u||} \psi_B\left(\frac{u}{||u||} \right) = \gamma \psi_B\left(u\right).
\end{multline*}
\end{proof}
 
\begin{proposition}\label{pospen}
	For any positive basis $N$, the augmenting functions $\psi_{\infty}^N$ and $\psi_{1}^N$ given in \eqref{neqn:2} and \eqref{neqn:3} respectively satisfy the conditions given in Theorem \ref{Thm1}.
\end{proposition}

\begin{proof} 
	
	Let $V= B^{\infty}_{\varepsilon} (0)$ be an open ball in the infinity norm with radius $\varepsilon >0$ centred at the origin. This is an appropriate open neighbourhood of $0$ for the purposes of Conditions 2 and 3 of Theorem \ref{Thm1}.
	\\
	
	\textbf{Condition 1: } $\psi(0) = 0$. \\
	If $u=0$ then $\text{\bf{n}}_i^{\top} u=0$ and therefore $\psi_{\infty}^{N} (u) =0$ and $\psi_1^{N} (u) = 0$, as required.\\
	
	\textbf{Condition 2: } $\psi(u) \geq \delta > 0, \forall u \not\in V$ for some positive scalar $\delta$. \\
	Using Theorem \ref{theorem:positive-dot-product}, for any $u \ne 0$ we have some $i$ such that $ \text{\bf{n}}_i^{\top} u >0$ and hence $\psi_{\infty}^{N} (u) >0$. Now define 
	\begin{eqnarray}\label{eq:pospen-con2-1}
		\delta &:=& \min_u \{ \max_{i=1,\dots,l} \{ \text{\bf{n}}_i^{\top} u \} \mid \|u \|_\infty = \varepsilon \}   > 0,
	\end{eqnarray}
	where $\delta > 0$ follows from the compactness of the $\varepsilon$- ball, the continuity of $u \mapsto \max_{i=1,\dots,l} \{ \text{\bf{n}}_i^{\top} u \}$, and Theorem~\ref{theorem:positive-dot-product}. For any $u \notin V$, the point $v := \varepsilon \frac{u}{\|u\|_{\infty}}$ is in $V$ and hence $\psi_{\infty}^N (v) \geq \delta >0$. Using the 
	positive homogeneity property we have 
	\begin{eqnarray*}
		&\frac{\varepsilon}{\|u\|_\infty}\psi_{\infty}^N (u) \geq \delta >0 \\
		\text{and so } &\quad \psi_{\infty}^N (u) \geq \delta \frac{\|u\|_{\infty}}{\varepsilon} \geq \delta >0,
	\end{eqnarray*}
	using the fact that $u \notin V$ means $\|u\|_{\infty} \geq \varepsilon$. This is the required inequality for $ \psi_{\infty}^N$.
	
	Apply Lemma \ref{lemma:positive-homogeneous-equivalence} to deduce that there exists a $\eta > 0$ such that:
	$$\psi_{1}^N (u) \geq \eta \psi_{\infty}^N (u) \geq \eta \delta >0 \text{ for all } u \notin V. $$
	$\eta \delta$ is also a positive scalar and so this is the required inequality for $\psi_{1}^N$.
	\\
	
	\textbf{Condition 3: } $\psi(u) \geq \gamma ||u||_\infty, \forall u \in V$ for some positive scalar $\gamma$.
	\\
	The property holds trivially for $u=0$. For any $u \in V \setminus \{0\}$,  the point $v := \varepsilon \frac{u}{\|u\|_{\infty}}$ is in $V$ and using the same $\delta$ as defined in \eqref{eq:pospen-con2-1} we have
	\begin{eqnarray*}
		&\frac{\varepsilon}{\|u\|_{\infty}}\psi_{\infty}^N (u) \geq \delta  >0 \\
		\text{and so } &\quad \psi_{\infty}^N (u) \geq \delta \frac{\|u\|_{\infty}}{\varepsilon} \geq \frac{\delta}{\varepsilon} \|u\|_{\infty}  >0,
	\end{eqnarray*}
	and so we may place $\gamma := \frac{\delta}{\varepsilon}>0$. This is the required inequality for  $\psi_{\infty}^N$.
	 
	As above, apply Lemma \ref{lemma:positive-homogeneous-equivalence} to deduce that there exists a $\eta > 0$ such that:
	$$\psi_{1}^N (u) \geq \eta \psi_{\infty}^N (u) \geq  \eta \gamma \| u \|_{\infty} >0. $$
	$\eta \gamma$ is also a positive scalar and so this is the required inequality for $\psi_{1}^N$.
\end{proof}

\begin{corollary}\label{pospen-corollary}
	Assume that $\zeta^{SIP}$ is feasible, its optimal value is finite, and the data which defines it is rational. Then the optimal value of the augmented Lagrangian dual problem $\zeta^{LD+}_\rho$ using an augmenting function of the form of (\ref{neqn:2}) or (\ref{neqn:3}) is equal to the optimal value of $\zeta^{SIP}$ for some finite $\rho$; that is,
	\begin{equation}\label{EqCloseDualGap}
	\zeta^{LD+}_\rho = \zeta^{LR+}_{\rho} (\omega_{LP}) = \zeta^{SIP}
	\end{equation}
	where $\omega_{LP}$ is the optimal multiplier of the linear programming relaxation of (\ref{NAC}). 
	In particular, this applies to our proposed augmenting function (\ref{neqn:1}).
\end{corollary}

\begin{proof}The equalities~\eqref{EqCloseDualGap} follow directly from Theorem \ref{Thm1} and Proposition \ref{pospen}. The last claim follows from the observation that (\ref{neqn:1}) may be represented as a function of the form of (\ref{neqn:3}), as demonstrated in (\ref{equation:phi-mu-represented-as-phi-1}).
\end{proof}

\begin{remark} \label{Remark8}
	Consider a positive basis $N=\left\{\bfemph{n}_{1},\dots , \bfemph{n}_{l}\right\}$. Each of the functions $g_i(u) = \max \{ \bfemph{n}_i^\top u, 0 \}$ is non-negative, positive homogeneous and finite valued, and these properties are preserved if multiple $g_i$s are summed, or their maximum is taken. By Theorem \ref{theorem:positive-dot-product}, for any non-zero $u$ there exists an index $i \in \{1,\dots,l\}$ such that $g_i(u)$ is strictly positive. Therefore, if every one of the $g_i$ functions is combined using a combination of summation and/or maximisation, the resulting function $g(u)$ will be strictly positive for all non-zero $u$. Applying Lemma \ref{lemma:positive-homogeneous-equivalence} to bound $g$ below by a positive multiple of $\psi_{\infty}^N$ (as $\psi_{1}^N$ was treated in Proposition \ref{pospen}) shows that this function $g(u)$ satisfies the conditions of Theorem \ref{Thm1}, and as such will close the duality gap if used as an augmenting function.	
\end{remark}

Remark \ref{Remark8} implies that we can not only construct $\psi_{\infty}^N$ and $\psi_{1}^N$ but also a wide variety of other augmenting functions from any given positive basis, depending on the order in which the maximisation and summation operations are applied to the $g_i$ functions.\\
	Furthermore, the sum or maximum of any two augmenting functions which satisfy the conditions of Theorem~\ref{Thm1} will itself satisfy the conditions of Theorem \ref{Thm1}, which yields further flexibility.

\begin{remark} \label{RemarkSurrogate}
	By using the positive basis $\{+e_i\}^{m}_{i=1} \cup \{ \sum_{i=1}^{m} - e_i\}$ or similar to define an augmenting function, we can obtain penalty terms analogous to the Lagrangian terms obtained through surrogate semi-Lagrangian relaxation.
\end{remark}

\section{Developing a penalty-based block Gauss-Seidel method} \label{sec3}

To exploit the potential for decomposability that this formulation presents, we consider a block Gauss-Seidel (GS) method approach.
In Section~\ref{SectGS}, we present a classical framework for GS methods as applied to nonlinear optimisation problems, and in Section~\ref{SectGSApp}, 
we show how it can be adapted to obtain solutions for SMIP.

\subsection{A block Gauss-Seidel method}\label{SectGS}

We consider the general problem given by
\begin{flalign}
\min_{x,z} \ &f(x,z) \label{GenProb}\\
\text{s.t.: } &x \in X, z \in Z. \nonumber  
\end{flalign}
We assume that $f$ is convex, but not necessarily differentiable. The sets $X$ and $Z$ are closed, but not necessarily convex.
The assumption that $x \in X$ and $z \in Z$ are taken from disjoint sets is adequate for our purposes, although block GS approaches have been studied in a more general setting where $(x,z)$ is taken from a set $K \subset \mathbb{R}^{n_x+n_z}$. (See~\cite{WendellHurter1976} for a treatment of the case where the constraint set is disjoint, and~\cite{Gorski2007} for the case where the constraint set may not be disjoint and developments are based on biconvexity assumptions.)
GS methods solve problem \eqref{GenProb} by separating it into two simpler problems. Given an iterate $(x^k, z^k)$, problem \eqref{GenProb} is solved with respect to $x$ for fixed $z = z^k$, yielding a new $x$-iterate $x^{k+1}$. Then, problem \eqref{GenProb} is solved with respect to $z$ for fixed $x = x^{k+1}$, yielding a new $z$-iterate $z^{k+1}$. 
In Algorithm~\ref{Alg1}, a formal listing of a block GS method applied to problem~\eqref{GenProb} is given.
\begin{algorithm}[H]
\caption{A block GS method}\label{Alg1}
\begin{algorithmic}[1] 
\State \textbf{initialise} $(x^0, z^0) \in X \times Z$
\For {$k = 1,\dots, k_{\text{max}}$}
    \State $x^{k} \gets \argmin_{x} \ \left\{f(x,z^{k-1}) :  x \in X \right\}$
    \State $z^{k} \gets \argmin_{z} \ \left\{f(x^k,z) :  z \in Z \right\}$
    \State $k \gets k + 1$
\EndFor
\State {\bf return} $(x^{k_{\text{max}}}, z^{k_{\text{max}}})$
\end{algorithmic}
\end{algorithm}

The sequence $\braces{(x^k,z^k)}$ generated by iterations of Algorithm~\ref{Alg1} has limit points when $X$ and $Z$ are compact. When $f$ is furthermore continuous and bounded from below over $X \times Z$, the limit points $(x^*,z^*) \in X \times Z$ are easily shown to be {\it partial minima}~\cite{WendellHurter1976}; that is, it holds that
\begin{flalign}
    &f(x^*,z^*) \leq f(x, z^*), \ \forall x \in X, \label{PartialOptX}\\
    &f(x^*,z^*) \leq f(x^*, z), \ \forall z \in Z. \label{PartialOptZ} 
\end{flalign}
This claim is formally stated in Proposition~\ref{PropGSPartMin}, whose proof is implicit from the developments of~\cite{WendellHurter1976} and is given here for the sake of completeness.
\begin{proposition}\label{PropGSPartMin}
For problem~\eqref{GenProb}, let $f$ be continuous and bounded from below, and let $X$ and $Z$ be compact. Then the  limit points $(x^*,z^*)$ of the sequence
$\braces{(x^k,z^k)}$ generated by iterations of Algorithm~\ref{Alg1} are partial minima.
\end{proposition}
\begin{proof}
We have by construction that $f(x^k,z^k) \le  f(x^k,z)$ for all $z \in Z$, and by the continuity of $f$, we have the second requirement $f(x^*,z^*) \leq f(x^*, z), \ \forall z \in Z$ for partial optimality.
To establish the first requirement~\eqref{PartialOptX}, assume for sake of contradiction that there is an $\bar{x} \in X$ for which $f(x^*,z^*) > f(\bar{x}, z^*)$. Due to the continuity of $f$, we have, for some infinite subsequence index set $\mathcal{K}$ such that $\lim_{k \to \infty, k \in \mathcal{K}} (x^k,z^k) = (x^*,z^*)$, the existence of $\gamma > 0$ such that $f(x^k,z^k) - f(\bar{x},z^k) > \gamma > 0$. Thus, $f(x^k,z^k) > f(\bar{x},z^k) + \gamma \ge f(x^{k+1},z^k) + \gamma \geq f(x^{k+1},z^{k+1}) + \gamma$, which would imply that $\lim_{k \to \infty} f(x^k,z^k) = -\infty$ since $\mathcal{K}$ is an infinite index set, so that $f$ is unbounded from below, a contradiction.
Therefore, $(x^*,z^*)$ must be a partial minimum for problem~\eqref{GenProb}.
\end{proof}
\begin{remark}
For practical purposes, we might approximate the satisfaction of~\eqref{PartialOptX} and~\eqref{PartialOptZ} through the $\epsilon \ge 0$ parameterised termination criterion
\begin{equation}\label{EqTermCond}
f(x^k,z^k) - f(x^{k+1}, z^{k+1}) \le \epsilon.
\end{equation}
\end{remark}

In the setting where $f$ is convex and differentiable, $X$ and $Z$ are nonempty, closed and convex, and $(x,z) \mapsto f(x,z)$ is inf-compact, 
it is well-known (see, for example,~\cite{Bertsekas1999,GS2000,Tseng2001}) that the limit points $(x^*,z^*)$ are optimal for problem~\eqref{GenProb}.
However, in the more general setting where $f$ is non-differentiable and/or $X$ and $Z$ are non-convex, it is well-known that a partial minimum need not be a global, or even a local, minimum. 
In what follows, we provide a few small examples to illustrate this suboptimal stabilisation, and to motivate heuristic features of our developed algorithm that can mitigate this unfortunate tendency.

\noindent
\medskip
{\bf Examples:}
\begin{enumerate}
\item Let problem~\eqref{GenProb} be specified so that $f(x,z) : \mathbb{R} \times \mathbb{R} \mapsto \mathbb{R}$ is defined to be $f(x,z) = 7x^2 + 10xz + 7z^2$, and let $X = Z=\braces{-2,-1,0,1,2}$. For $(x^0,z^0)=(2,-2)$, the application of Algorithm~\ref{Alg1} leads immediately to the one limit point $(\bar{x},\bar{z})=(1,-1)$. We have $f(1,-1)=4$, but $f(0,0)=0$, so $(\bar{x},\bar{z})=(1,-1)$ is not optimal. Note here that $f$ is convex and continuously differentiable, but the constraint set $X \times Z$ is nonconvex due to the integer restriction, and this is the reason that the limit point was not guaranteed to be optimal.
\item Let problem~\eqref{GenProb} be specified so that $f(x,z) : \mathbb{R} \times \mathbb{R} \mapsto \mathbb{R}$ is defined to be $f(x,z) = -2x - z + \rho \left| x-z \right|$, $X=[-2,3]$, and $Z=[0,5]$.
For $\rho \in [0,1)$, the optimal solution is $(x^*,z^*)=(3,5)$. For $\rho = 1$, the optimal solutions are taken from $(x^*,z^*) \in \braces{3} \times [3,5]$, and for $\rho > 1$, the optimal solution is $(x^*,z^*) = (3,3)$.
\begin{enumerate}
\item When applying the GS approach of Algorithm~\ref{Alg1} with $\rho \in (0,1)$, the resulting sequence stabilises after one iteration at the optimum $(x^*,z^*) = (3,5)$ for any feasible starting point. 
\item For $\rho=1$ with $z^0 \ge 3$, we have after half an iteration $(x^1,z^0)=(3,z^0)$ which is an optimum solution, and the remaining updates stay at some optimal solution $(x^*,z^*) \in \braces{3} \times [3,5]$. For $\rho=1$ with starting point $z^0 < 3$, we have $x^1=3$ and $z^1 \in [3,5]$ and so stabilisation at an optimal solution also occurs. 
\item For $\rho > 1$ with $z^0 \ge 3$, we have $(x^1,z^1)=(3,3)$, which is optimal. However, for $\rho > 1$ with $z^0 < 3$, we have $x^1 = z^0$ and $z^1=z^0$, so that stabilisation occurs at $(\bar{x},\bar{z})=(z^0,z^0)$, which is not optimal. 
\end{enumerate}
\item Let problem~\eqref{GenProb} be specified so that $f(x,z) : \mathbb{R}^3 \times \mathbb{R}^3 \times \mathbb{R}^3 \mapsto \mathbb{R}$ is defined to be
$$
f(x,z) = 2x_{1,1}-1x_{1,2}-2x_{1,3} -2x_{2,1} -1 x_{2,2} + 2x_{2,3} + \rho \sum_{i=1,2} \sum_{j=1,2,3} \left| x_{i,j}-z_j \right|,
$$
and let $X$ and $Z$ be defined so that
$$X = \braces{(x)_{i,j} : \sum_{j=1}^3 x_{i,j} \le 1 \;\text{for}\; i=1,2;\; x_{i,j} \in \braces{0,1}\;\text{for}\; i=1,2,\; j=1,2,3},$$ and $Z = \braces{0,1}^3$.
For $\rho \to \infty$ (simulating the enforcement of constraints $x_{i,j}=z_j$ for $i=1,2$ and $j=1,2,3$) we have the optimal solution 
\[
(x^*,z^*)=\left([(0,1,0),(0,1,0)],0,1,0\right).
\]
If such constraints are altogether ignored ($\rho=0$), then the optimal $x$-component  is $x^*=\left( (0,0,1),(1,0,0) \right)$. 
This behaviour would only change at the threshold $\rho=1$. For $\rho > 1$, the optimal solution would be  
$(x^*,z^*) =   ([(0,1,0),(0,1,0)],(0,1,0))$. 
\begin{enumerate}
\item Now we consider what happens when the GS approach of Algorithm~\ref{Alg1} is applied. Let $z^0=(0,0,0)$. Starting with a small penalty such as $\rho=0.5$, we have
\begin{align*}
x^1=\left( (0,0,1),(1,0,0) \right) \quad \text{and} \quad z^1 \in \braces{(0,0,0), (0,0,1), (1,0,0), (1,0,1)},
\end{align*}
where there is more than one way to choose $z^1$. If, for example, we make it a policy to choose $z$ by some bitwise lexicographical rule, then we choose $z^1=(0,0,0)$.
Keeping this same penalty $\rho=0.5$, we find that stabilisation has occurred, where $x^k=x^1$ and $z^k=z^1$ for $k \ge 1$. 
If we increase the penalty value to $\rho=2$ for iteration $k=2$, then 
we have the stabilisation $x^1=\left( (0,0,0),(0,0,0) \right)$ and $z^1=(0,0,0)$, which is suboptimal (and $\rho=2$ is the threshold for this change in stabilisation to occur).

If, instead, the $z$ update is chosen by a reverse-lexicographic rule, so that $z^1=(1,0,1)$, then we have immediate stabilisation with
$$(x^k,z^k)=\left( (0,0,1),(1,0,0),(1,0,1)\right)$$ for all $k \ge 1$ for all $\rho > 0$. (Notice that no matter how large the penalty is, consensus is not achieved in the GS setting. That is, without additional restriction on how $z$ is updated, the optimal $z$ update may be chosen to always correspond to a consensus solution that is infeasible for both scenarios. In practice, we would need a rule to insure that the $z$ update is chosen to satisfy $\sum_{j=1}^3 z_j \le 1$ to match with the constraints in the $x$ update subproblems.)
\item The shortcomings of the above GS approach motivate the introduction of more precision in how the consensus discrepancies are penalised, where $f$ is redefined to be
\begin{equation*}
f(x,z) = 2x_{1,1}-1x_{1,2}-2x_{1,3} -2x_{2,1} -1 x_{2,2} + 2x_{2,3} + \sum_{i=1,2} \sum_{j=1,2,3} \rho_{i,j}\left| x_{i,j}-z_j \right|.
\end{equation*}
That is, instead of one scalar $\rho$, we have term-specific $\rho_{i,j}>0$ for each $i=1,2$ and $j=1,2,3$. We start as before with $z^0=(0,0,0)$, and let $\rho_{i,j}=0.5$ for each $i=1,2$ and $j=1,2,3$. Assuming lexicographic rule in choosing $z$, we have as before
$$
x^1=\left( (0,0,1),(1,0,0) \right) \quad \text{and} \quad z^1 =(0,0,0),
$$
and this is stable if the penalty does not change.
Now increase $\rho_{1,3}=\rho_{2,1}=1$, and we have 
$$
x^2=\left( (0,0,1),(1,0,0) \right) \quad \text{and} \quad z^2 =(1,0,1),
$$
and this is stable if the penalty does not change. Increasing  $\rho_{1,1}=\rho_{2,3}=1$, we have again
$$
x^3=\left( (0,0,1),(1,0,0) \right) \quad \text{and} \quad z^3 =(0,0,0),
$$
and this is stable. But once we again increase $\rho_{1,3}=\rho_{2,1}=2$, we have
$$
x^4=\left( (0,1,0),(0,1,0) \right) \quad \text{and} \quad z^4 =(0,1,0),
$$
which is optimal for the original problem.
\end{enumerate}
\end{enumerate}

The last example suggests that there may be no fixed ideal penalty in a GS setting that will lead to both a closing of the duality gap and avoiding the nonoptimal stationarity due to GS iterations. The penalty must vary in a manner that takes the component-wise consensus status into consideration. Any sensible heuristic for varying the penalties would have all penalties start small (but nonzero), and increase carefully, in a "fine-tuned" manner so as to ``suggest'' a temporary fixing of certain components of $x$ to the current fixed values of the corresponding components of $z$. The strength of suggestion for each component is always relative to the other components as the magnitude of each component-wise penalty is relative to the magnitude of the other component-wise penalties. 

An approach based on such an idea where some subset of variables is subject to ``suggested'' fixing with strength of suggestion determined by the penalties would be of a ``soft'' combinatorial nature. 
This is in contrast with a ``hard'' combinatorial approach that might be based on the idea of choosing some subset of integer variables  at each iteration and simply fixing each one to some constant feasible value while conducting a minimisation over the unfixed variables. The algorithm to be presented later is of a soft combinatorial nature.

\subsection{Adapting block Gauss-Seidel method to solve SMIPs using Penalty functions}\label{SectGSApp}

In this section we present how block GS method can be used to obtain solutions for SMIP problems. The approach will rely on the delayed calculation of variable $z$, which will in turn allow us to obtain a decomposed version of the problem. To do such, let us first explicitly state $\zeta^{LR+}_\rho(\omega)$ as
\begin{flalign}
\zeta^{LR+}_\rho(\omega) :~ \min_{x,y,z} &\sum_{s \in S}p_s(\left[c +\omega \right]^\top x_s + q_s^\top y_s) \nonumber\\
&+ \sum_{s \in S}\underline{\rho}_s^\top[x_s - z]^- +  \sum_{s \in S}\overline{\rho}_s^\top [z - x_s]^- \label{phi-mu} \\ 
\text{s.t.: } &x_s \in X, \ \forall s \in S  \\
&y_s \in Y_s(x_s), \ \forall s \in S. \label{EqMIPPenalty}
\end{flalign} 
%
The following proposition will become useful in the following derivations.
\begin{proposition} \label{Prop1}
For problem~\eqref{phi-mu}-\eqref{EqMIPPenalty} with any $\omega=\tilde{\omega} \in \mathbb{R}^{n_x \times |S|}$, there exists a finite $\rho^*(\tilde{\omega})$ such that $\zeta^{LR+}_{\rho^*}(\tilde{\omega}) = \zeta^{SIP}$.
\end{proposition}

\begin{proof}
The penalty terms in~\eqref{phi-mu} result from the evaluation of $\psi_\rho\left( (x_s - z)_{s \in S} \right)$ with $\psi_\rho$ as defined in~\eqref{neqn:1}.
		Thus, by Corollary \ref{pospen-corollary}, the requirements of Theorem \ref{Thm1} are satisfied. 
		Now one can rely on Remark \ref{BolandEberhard} to make a free choice of $\omega$. 
\end{proof} \\

Proposition \ref{Prop1} enables us to make the choice of $\omega = 0$, which leads to 
\begin{flalign}
\zeta^{LR+}_\rho(0) :~ \min_{x,y,z} &\sum_{s \in S}p_s(c^\top x_s + q_s^\top y_s) + \sum_{s \in S}\underline{\rho}_s^\top[x_s - z]^- +  \sum_{s \in S}\overline{\rho}_s^\top[z - x_s]^- \\ \label{phi-mu2}
\text{s.t.: } &x_s \in X, \ \forall s \in S  \\
&y_s \in Y_s(x_s), \ \forall s \in S.
\end{flalign} 

The block GS method for solving $\zeta^{LR+}_\rho(0)$ proceeds as follows. Let 
\begin{eqnarray*}
\phi^\rho(x,y,z,\rho) &:=&  \sum_{s \in S}p_s\phi^\rho_s(x_s,y_s,z,\mu_s), \\ 
\end{eqnarray*}
where $$\phi^\rho_s(x_s,y_s,z,\mu_s):= \left\{  c^\top x_s + q_s^\top y_s + \underline{\mu}_s^\top[x_s - z]^- +  \overline{\mu}_s^\top[z - x_s]^- \right\}$$
and
$(\underline{\mu}_s,\overline{\mu}_s):= (\frac{1}{p_s}\underline{\rho}_s,\frac{1}{p_s}\overline{\rho}_s)$ for each $s \in S$.
For a given $\rho^k_s = (\underline{\rho}^k_s,\overline{\rho}^k_s)_{s \in S}$ and an initial $z^{0,0}$, the proposed method will iterate between the solution of following $l=0,1,\dots,l_{\text{max}}$ subproblems:
\begin{flalign}
(x^{k,l+1}, y^{k,l+1})_{s \in S} \gets \argmin_{x,y} \ &\phi^\rho(x,y,z^{k,l},\rho^k)\nonumber\\
\text{s.t.: } 
&x_s \in X, \ \forall s \in S  \nonumber \\
&y_s \in Y_s(x_s), \ \forall s \in S, \nonumber
\end{flalign}
and
\begin{flalign}
z^{k,l+1} \gets \argmin_{z} \ &\phi^\rho(x^{k,l+1},y^{k,l+1},z,\rho^k),
\end{flalign}
followed by $l = l+1$ and successive repetition until \emph{partial convergence} is approximately achieved in the sense of~\eqref{EqTermCond}. In this context, \textit{partial convergence} can be interpreted as having
$$\phi^\rho(x^{k,l},y^{k,l},z^{k,l},\rho^k) - \phi^\rho(x^{k,l+1},y^{k,l+1},z^{k,l+1},\rho^k) \leq \epsilon, $$
given a threshold $\epsilon \ge 0$. 

At last, if the current primal infeasibility level, giving by a residual measure such as $||x^{k,l} - z^{k,l}||_2^2$, is not acceptable for a $\epsilon$ threshold, the set of penalties $\rho^k = (\underline{\rho}^k, \overline{\rho}^k)$ are then updated to $\rho^{k+1} = (\underline{\rho}^{k+1}, \overline{\rho}^{k+1})$ and the process is repeated for iteration $k+1$.

\section{Computational Implementation Aspects} \label{sec4}

Two remarkable features can be exploited in the design of an algorithm based on this idea. First, scenario-wise separability is straightforwardly achieved in the calculation of $(x^{k,l+1}, y^{k,l+1})$. This means that instead of solving one large mixed-integer linear programming (MILP) problem in this update step, we can solve several small MILP problems instead, which is typically more efficient due to the exponential nature of the branch-and-cut-based methods used to solve them.  


To formulate the $(x^{k,l}, y^{k,l})$-update
\begin{flalign}
(x^{k,l+1}, y^{k,l+1})_{s \in S} \in \argmin_{x,y} \ &\phi^\rho(x,y,z^{k,l},\rho^k)\nonumber\\
\text{s.t.: } 
&x_s \in X, \ \forall s \in S  \nonumber \\
&y_s \in Y_s(x_s), \ \forall s \in S, \nonumber
\end{flalign}
it is necessary to explicitly represent the function $[ \ \cdot \ ]^-$. To do so, we consider an equivalent reformulation of the problem given by 
\begin{flalign}
\phi^\rho_s(x_s^{k,l+1},y_s^{k,l+1},z^{k,l},\mu_s)=~ \min_{x,y,\underline{w},\overline{w}} \ &c^\top x_s + q_s^\top y_s + {(\underline{\mu}^k_s)}^\top \underline{w}_s + {(\overline{\mu}^k_s)}^\top \overline{w}_s \nonumber\\
\text{s.t.: } 
&\underline{w}_s \geq 0, \quad 
\underline{w}_s \geq z^{k,l} - x_s \nonumber \\
&\overline{w}_s \geq 0, \quad 
\overline{w}_s \geq x_s - z^{k,l} \nonumber \\
&x_s \in X, \quad 
y_s \in Y_s(x_s). \nonumber
\end{flalign}
Second, the calculation of 
\begin{equation}
z^{k,l+1} \in \argmin_{z} \ \phi^\rho(x^{k,l+1},y^{k,l+1},z,\rho^k) \label{neqn:124}
\end{equation}
may be performed by computing 
$$
z^{k,l+1} \in \argmin_{z} \ \zeta^\rho(x^{^{k,l+1}} ,z,\rho^k),
$$
where the penalty function $(x,z) \mapsto \zeta^\rho(x,z,\rho)$ is defined by
$$
\zeta^\rho(x ,z,\rho):=\psi_\rho( (x_s-z)_{s \in S} ) = \sum_{s \in S}\left( \underline{\rho}_s^\top [x_s - z]^- + \overline{\rho}_s^\top [z - x_s]^-\right).
$$
The last displayed problem can be solved using the following equivalent mathematical programming formulation:
\begin{flalign}
\zeta^\rho(x^{k,l+1} ,z^{k,l+1},\rho^k) =~ \min_{z,\underline{w},\overline{w}} \ &\sum_{s \in S} (\underline{\rho}^k_s)^\top \underline{w}_s + (\overline{\rho}^k_s)^\top \overline{w}_s \label{zStepStart} \\
\text{s.t.: } 
&\underline{w}_s \geq 0, \ \forall s \in S, \quad 
\underline{w}_s \geq z- x_s^{k,l+1} , \ \forall s \in S \\
&\overline{w}_s \geq 0, \ \forall s \in S, \quad 
\overline{w}_s \geq x_s^{k,l+1} - z , \ \forall s \in S. \label{zStepFinish}
\end{flalign}

When the $x$ components are all restricted to take binary values, it is possible to show that the calculation of $z^{k,l+1}$ can be performed in the following closed form where each component of $z^{k,l+1}$ always takes binary value. In that case, its optimal solution is given by
\begin{equation} \label{zExact}
z^{k,l+1}_i = \begin{cases}
               1, \text{ if } \sum_{s \in S}(1-x^{k,l+1}_{s,i})\underline{\rho}^k_s < \sum_{s \in S}x^{k,l+1}_{s,i}\overline{\rho}^k_s \\
               0, \text{ if } \sum_{s \in S}(1-x^{k,l+1}_{s,i})\underline{\rho}^k_s > \sum_{s \in S}x^{k,l+1}_{s,i}\overline{\rho}^k_s \\
               \text{either } 0 \text{ or } 1, \quad\quad\quad\quad \text{ otherwise}
              \end{cases}, \ i = 1, \dots, n_x.
\end{equation}
The cases in which we have a tie might require "flipping a coin" for deciding on the value for $z^{k,l+1}$, as it becomes a case of multiple minima. The existence of multiple minima can be 
better understood from the following explicit form of the solution for the general case. 
In the following proposition, we assume $Z$ is a closed convex set, so that no explicit integrality constraints are enforced.

\begin{proposition}
Suppose a set of scenario dependent solutions $ (x_s)_{s\in S}$, 
where $x_s =(x_{s,i})_{i=1,\dots,n_x} $, are given 
and $z:=(z_i)_{i=1,\dots, n_x}$.  For each $i \in \{1,\dots,n_x\}$ define 
\begin{eqnarray*}
I^{+} (z_i) & := & \{s \in S \mid x_{s,i} > z_i \} \\
I^{-} (z_i) & := & \{s \in S \mid x_{s,i} < z_i\} \\
I^{0} (z_i) & := & \{s \in S \mid x_{s,i} = z_i\}
\end{eqnarray*}
Then $z_i$ solves problem~\eqref{neqn:124} given fixed $(x_s)_{s\in S}$ if and only if
\begin{equation}
\sum_{s \in I^{+} (z_i)} \overline{\rho}_{s,i} - \sum_{s \in I^{-} (z_i)} \underline{\rho}_{s,i}  \in \left[ -\sum_{s \in I^{0} (z_i)}\overline{\rho}_{s,i},\sum_{s \in I^{0} (z_i)}\underline{\rho}_{s,i}  \right].  \label{neqn:122}
\end{equation}
\end{proposition}

\begin{proof}
	The index $s$ term of the penalty function $\zeta^\rho$ may be written as 
\begin{eqnarray*}
&&\zeta^\rho_s((x_s)_{s\in S} ,z,\rho) \\
&:=& \sum_{i=1}^{n_x} \left[ \sum_{s \in I^{+} (z_i)} \overline{\rho}_{s,i}  \max\{0, x_{s,i} - z_i\}
+  \sum_{s \in I^{-} (z_i)} \underline{\rho}_{s,i} \max\{0, z_i - x_{s,i}\} \right].
\end{eqnarray*}
As this is separable in the variables $(z_1, \dots, z_{n_x})$, its subdifferential is defined as the cross product of intervals, one for each component $i$. Thus, the necessary and sufficient condition 
$$0 \in \partial \zeta^\rho_s((x_s)_{s\in S} , z, \rho),$$ can be equivalently stated as 
$$0 \in \partial_{z_i} \zeta^\rho_s((x_s)_{s\in S} ,z_i,\rho),$$
for each $i =1,\dots,n_x$, which is given by:
\begin{align*}
0 &\in  \sum_{s \in I^{-} (z_i )} \underline{\rho}_{s,i} -\sum_{s \in I^{+}(z_i)} \overline{\rho}_{s,i}  + \sum_{s \in I^{0} (z_i)}\left[-\overline{\rho}_{s,i} , \underline{\rho}_{s,i}\right] \\
&= \sum_{s \in I^{-} (z_i )} \underline{\rho}_{s,i} -\sum_{s \in I^{+}(z_i)} \overline{\rho}_{s,i}  + \left[-\sum_{s \in I^{0} (z_i)}\overline{\rho}_{s,i} , \sum_{s \in I^{0} (z_i)}\underline{\rho}_{s_i,i}\right].
\end{align*}
which in turn is equivalent to (\ref{neqn:122}).
\end{proof}

We now consider how to update the penalty parameters $\rho^k$. 
A simple strategy is
\begin{flalign*}
\underline{\rho}_s^{k+1} = \underline{\rho}_s^{k} + \gamma [x_s^{k,l} - z^{k,l}]^- \\ 
\overline{\rho}_s^{k+1} = \overline{\rho}_s^{k} + \gamma [z^{k,l} - x_s^{k,l}]^-. 
\end{flalign*}
%
By doing so, we are reinforcing the penalties associated with the respective discrepancies. In other words for each $i = 1,\dots,n_x$:
$$
\underline{\rho}_{s,i}^{k+1}  = \begin{cases}
\underline{\rho}_{s,i}^{k} + \gamma (z^{k,l}_i - x_{s,i}^{k,l}), &\text{ if } x_{s,i}^{k,l} < z^{k,l}_i\\
\underline{\rho}_{s,i}^{k}, &\text{ if } x_{s,i}^{k,l} \geq z^{k,l}_i
\end{cases}$$
$$\overline{\rho}_{s,i}^{k+1} = \begin{cases}
\overline{\rho}_{s,i}^{k}  + \gamma (x_{s,i}^{k,l} - z^{k,l}_i), &\text{ if } z^{k,l}_i < x_{s,i}^{k,l}\\
\overline{\rho}_{s,i}^{k}, &\text{ if } z^{k,l}_i \geq x_{s,i}^{k,l}
\end{cases}
$$

\begin{remark}
The update in $\rho^{k+1}$ has the effect of changing the left hand side of  (\ref{neqn:122}) at the next iteration by the amount:
\begin{equation}
\Delta^{k+1}_i := \gamma \left[\sum_{s \in I^{+}(z^k_i)}[z^k_i - x_{s,i}^k]^{-} - \sum_{s\in I^{-} (z^k_i)} [x^k_{s,i} - z^k_i]^{-}    \right], \label{neqn:123}
\end{equation}
for each $i=1,\dots, n_x$. 
If the addition of this factor ensures the sum in left hand side of (\ref{neqn:122}) at iteration $k+1$ exits the interval 
$\left[ -\sum_{s \in I^{0} (z_i)}\overline{\rho}_{s,i},\sum_{s \in I^{0} (z_i)}\underline{\rho}_{s,i}  \right]$ associated with the prior choice of $z^k_i =x^k_{s,i}$ then we would be forced to choose new consensus values $z^{k}_i$ in order to re-establish the satisfaction of the optimality condition (\ref{neqn:122}). In doing so, a reassignment of the index sets $I^{+} (z_i^k)$, $I^{-} (z_i^k)$, and $I^{0} (z_i^k)$ is effected. As intuition would suggest, the optimality condition (\ref{neqn:122}) is more easily satisfied when $s \in I^{0}(z_i^k)$ for large $\underline{\rho}_{s,i}$ and $\overline{\rho}_{s,i}$, as this makes the target interval $\left[ -\sum_{s \in I^{0} (z_i)}\overline{\rho}_{s,i},\sum_{s \in I^{0} (z_i)}\underline{\rho}_{s,i}  \right]$ larger.
\end{remark}


To effect a gradual increase in the terms $\Delta^k$ in an attempt to improve convergence with the satisfaction of the NAC, we considered an increasing multiplier factor to $\psi_\rho$ given by $\beta^{(k-1)}-1$ (where $(k-1)$ represents an exponent and not an iteration index). In other words, we consider the objective
at a given iteration $k$ as being 
$$\phi^{\rho,k}(x_s,y_s,z,\omega):=\sum_{s \in S}p_s(c^\top x_s + q_s^\top y_s) + (\beta^{(k-1)}-1)\left[\sum_{s \in S}\underline{\rho}_s^T[x_s - z]^- +  \sum_{s \in S}\overline{\rho}_s^T[z - x_s]^-\right] \\ .$$

Combining what have been exposed so far, one first algorithmic approach consists of the following setting presented in Algorithm~\ref{Alg2}.
\begin{algorithm}[H]
\caption{Alternating direction method for SMIP}\label{Alg2}
\begin{algorithmic}[1] 
\State \textbf{initialise} $\rho^0=(\underline{\rho}^0, \overline{\rho}^0), \hat{z}^{0}, 
\epsilon, \gamma, \beta, l_{\text{max}}, k_{\text{max}}$
\For {$s \in S$}    
    \State $\hat{x}^0_s \gets \argmin_{x,y} \ \braces{\phi^{\rho,1}    (x_s,y_s,\hat{z}^{0},\rho^0) : x_s \in X, y_s \in Y_s(x_s)}$ \label{Alg2:Line3}
\EndFor
\For {$k = 1,\dots, k_{\text{max}}$}
    \State $x^{k,0} \gets \hat{x}^{k-1}$
    \State $z^{k,0} \gets \hat{z}^{k-1}$
    \For {$l = 1, \dots, l_{\text{max}}$}
        \For{$s \in S$}
           \State $(x_s^{k,l},y_s^{k,l}) \gets \argmin_{x,y} \ \braces{\phi^{\rho,k}(x_s,y_s,z^{k,l-1},\rho^k) : x_s \in X, y_s \in Y_s(x_s)}$
        \EndFor     
        \State $z^{k,l} \gets \argmin_{z} \ \phi^{\rho,k}(x^{k,l},y^{k,l},z,\rho^k)$ \label{zUpdate}
            
          \State $\Gamma \gets \phi^{\rho,k}(x^{k,l-1},y^{k,l-1},z^{k,l-1},\rho^k) - \phi^{\rho,k}(x^{k,l},y^{k,l},z^{k,l},\rho^k)$
        \If {$\Gamma \leq \epsilon$ \textbf{ or } $l=l_\text{max}$}  
        
            \State $(\hat{x}_s^{k},\hat{y}_s^{k}) \gets (x_s^{k,l},y_s^{k,l})$ for all $s \in S$
            \State $\hat{z}^{k} \gets z^{k,l}$
            \State \textbf{break} 
        \EndIf
        \State $l \gets l + 1$ 
    \EndFor
    \If{$||\hat{x}^{k} - \hat{z}^{k}||_2^2 \leq \epsilon$ \textbf{ or } $k = k_{\text{max}}$}
        \State \textbf{return } $((\hat{x}_s^k, \hat{y}_s^k)_{s\in S}, \hat{z}^k)$    
    \Else \quad 
        \State $\underline{\rho}_s^{k} = \underline{\rho}_s^{k-1} + \gamma [\hat{x}_s^{k} - \hat{z}^{k}]^-$ for all $s \in S$
        \State $\overline{\rho}_s^{k} = \overline{\rho}_s^{k-1} + \gamma [\hat{z}^{k} - \hat{x}_s^{k}]^-$ for all $s \in S$
        
    \EndIf
    \State $k \gets k + 1$
\EndFor
\end{algorithmic}
\end{algorithm}

%

\section{Experimental setting} \label{sec5}

In this section we describe the computational experiments performed to assess the performance of the proposed approach. To evaluate the performance of the proposed method, we tested its efficacy on three distinct classes of problems from literature, namely the capacitated facility location problems (CAP) from \cite{boduretal2014}, the dynamic capacity allocation problems (DCAP) available in \cite{ahmedgarcia2003}, and the server location under uncertainty problems (SSLP) first introduced in \cite{ntaimosen2005}. To provide a more solid base of comparison, 50 random instances of two problems from each class were generated. 

The CAP problems are two-stage SMIP problems with pure binary first- and second-stage variables arising in the context of network design problems. We selected the instances coded as 101 and 111 in \cite{boduretal2014}, considering random samples of 100 scenarios from a list of 5000 scenarios available. 

The DCAP problems are two-stage SMIP problems arising in dynamic capacity acquisition and allocation under uncertainty. All problem instances have mixed-integer first-stage variables and pure binary second-stage variables. We selected the instances coded as 233 and 342 (which encodes the number of resources, tasks, and periods, respectively), considering random samples of 100 scenarios from the original 500 available.

The SSLP problems are two-stage SMIP problems arising in server location under uncertainty. The problems have pure binary first-stage variables and mixed-binary second-stage variables. We considered the instances coded as 5-50 and 10-50 (which encode the number or servers and the number of clients, respectively) with 100 scenarios that were randomly generated according to the guidelines provided in \cite{ntaimosen2005}. 

To compare and benchmark the performance of the proposed approach against a known quantity we have implemented the Progressive Hedging (PH) algorithm, which was originally proposed by \cite{RockafellarPH1991} and, as previously discussed, has been widely used as an heuristic approach to solve SMIP problems. The PH algorithm is stated in Algorithm \ref{Alg3} for the sake of completeness. In this algorithm, $$L_\rho^s(x_s,y_s,z,\omega_s) := (c + \omega_s^k)^\top x_s + q_s^\top y_s + \frac{\rho}{2}||x_s - z||_2^2,$$ 
which means that Line \ref{PHSubProb} comprises the solution of $|S|$ mixed-integer quadratic subproblems at each iteration $k$. The analogous step in PGBS algorithm requires us to solve typically less difficult MIP problems instead. 

Another advantage of PGBS in the context of SMIP problems is that $z^{k,l+1}$ tends to (in most cases) satisfy the integrality constraints of the problem. This is in contrast with the consensus value computed with the averaging of PH (Line~\ref{PHAveraging} in Algorithm~\ref{Alg3}), which tends to steer the consensus value away from integral values. 
In the case of binary variables, the PH averaging computation of the consensus $z^{k,l+1}$ is especially prone to producing many fractional valued components which can lead to episodic cycling in binary values set in the assignment of scenario specific variables. 

\begin{algorithm}[H]
\caption{Progressive Hedging for SMIP}\label{Alg3}
\begin{algorithmic}[1] 
\State \textbf{initialise} $\rho, (\omega^0_s)_{s \in S},\epsilon, k_{\text{max}}$
\For {$s \in S$}    
   \State $x^0_s \gets \argmin_{x,y} \ \braces{c^\top x_s + q_s^\top y_s : x_s \in X, y_s \in Y_s(x_s)}$ \label{Alg3:Line3}
\EndFor
\State $z^0 \gets \sum_{s}p_s x^0_s$ 
\For {$k = 1,\dots, k_{\text{max}}$}
    \For{$s \in S$}
        \State $(x_s^{k},y_s^{k}) \gets \argmin_{x,y} \ \braces{L^s_{\rho} (x_s,y_s,z^k,\omega^{k-1}) : x_s \in X, y_s \in Y_s(x_s)}$ \label{PHSubProb}
    \EndFor     
    \State $z^{k} \gets \sum_{s}p_s x^k_s$ \label{PHAveraging}
    \If{$||x^{k} - z^{k-1}||_2^2 \leq \epsilon$ \textbf{ or } $k = k_{\text{max}}$}
        \State \textbf{return } $((x_s^k, y_s^k)_{s\in S}, z^k)$    
    \Else \quad 
        \State $\omega_s^{k} \gets \omega_s^{k-1} + \rho (x_s^{k} - z^{k}), \ \forall s \in S$
    \EndIf
     \State $k \gets k + 1$
\EndFor
\end{algorithmic}
\end{algorithm}

In the PGBS experiments, the parameters were chosen from $\beta \in \{1.25, 1.11\}$,  and $\gamma \in \{0.5\rho^0, \rho^0\}$. Three different initial values for $\rho^0$ were used in both the PGBS and PH experiments. In the Progressive Hedging algorithm, dual multipliers were initialised as $0$ and the penalty parameter ($\rho$) was set to $\rho=\rho^0$. The parameter $z^{0,0}$ has been initialised according to the solution $x^0_s$ (from Line~\ref{Alg2:Line3} in Algorithm~\ref{Alg2} and Line~\ref{Alg3:Line3} in Algorithm~\ref{Alg3}). For Algorithm~\ref{Alg2}, we initialised $z^{0,0}$ as being
\begin{equation}
z_i^{0,0} = \left\lceil \sum_{s \in S}p_s x^0_{i,s} \right\rfloor, \ \forall i= 1\dots,n_x
\end{equation}
for all components restricted to be integer variables, where $\lceil \ \cdot \ \rfloor$ denotes rounding operation. For the components without integrality restrictions, we have dropped the rounding operator. In case of PH (Algorithm~\ref{Alg3}), all components were calculated by dropping the rounding operator.  

As both CAP and SSLP problems have pure binary first-stage variables, we have used \eqref{zExact} to perform the step depicted in Line~\ref{zUpdate} of Algorithm~\ref{Alg2}. For DCAP, we relied on solving \eqref{zStepStart}-\eqref{zStepFinish} explicitly. 

A time limit of 1000 seconds and termination condition of $\epsilon = 10^{-3} $ was used for both methods. A total of 300 ($3\times2\times50$) instances were solved with three parameter choices for PH (different choices of $\rho^0$) and 12 combinations of parameter choices for PBGS (different choices of $\rho^0$, $\beta$ and $\gamma$). The computational experiments were performed on a Intel i7 CPU with 3.40GHz and 8GB of RAM. All methods have been implemented using AIMMS 3.14 and all subproblems have been solved using CPLEX 12.6.3 with its standard configuration. 

\subsection{Numerical results}

A summary of the computational results is presented in Figures \ref{Fig:CAP} to \ref{Fig:SSLP}, which depicts the average computational time and objective value difference for the 50 instances considered for both PH and PBGS in all parameter settings that have been tested. 

The blue bars indicate the average wall clock times for both methods. We highlight that the instances in which PH terminated due to the time limit of 1000s have been removed from the average calculations, these being treated as outliers. The green line shows the average objective value relative difference, which is calculated as
$$\frac{1}{N}\sum_{i=1}^N\frac{z^i_{PBGS} - z^i_{PH}}{z^i_{PH}},$$ 
where $z^i_{PBGS}$ and $z^i_{PH}$ are the objective function values obtained for the solutions returned by PBGS and PH for instance $i$, respectively, and $N$ is the total number of instances considered for average value calculations. To obtain $z^i_{PBGS}$ and $z^i_{PH}$, we used the last solution returned by both methods and evaluated it a posteriori. For the cases in which PH returned solutions that were infeasible in regard to integrality restrictions (typically those obtained when the algorithm stopped due to the time criterion), rounding has been performed to recover a feasible solution to be evaluated when applicable.  

\begin{figure}[H]
    \begin{subfigure}[b]{0.5\textwidth}
        \includegraphics[width = \textwidth]{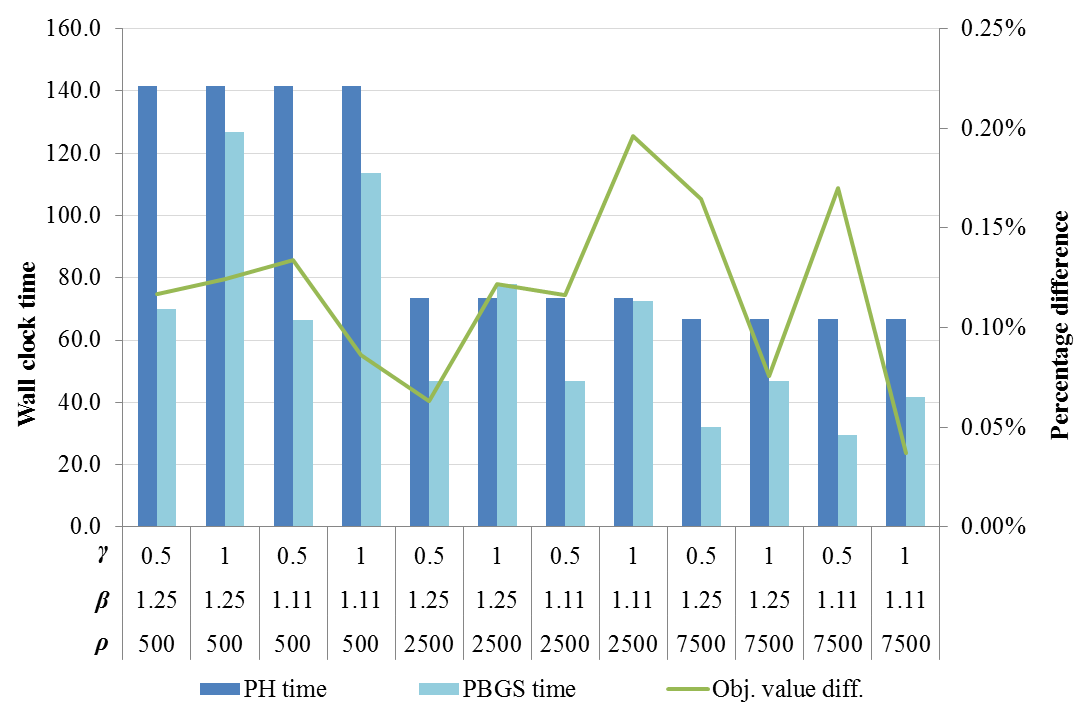}
        \caption{CAP101} 
    \end{subfigure}
    ~
    \begin{subfigure}[b]{0.5\textwidth}
        \includegraphics[width = \textwidth]{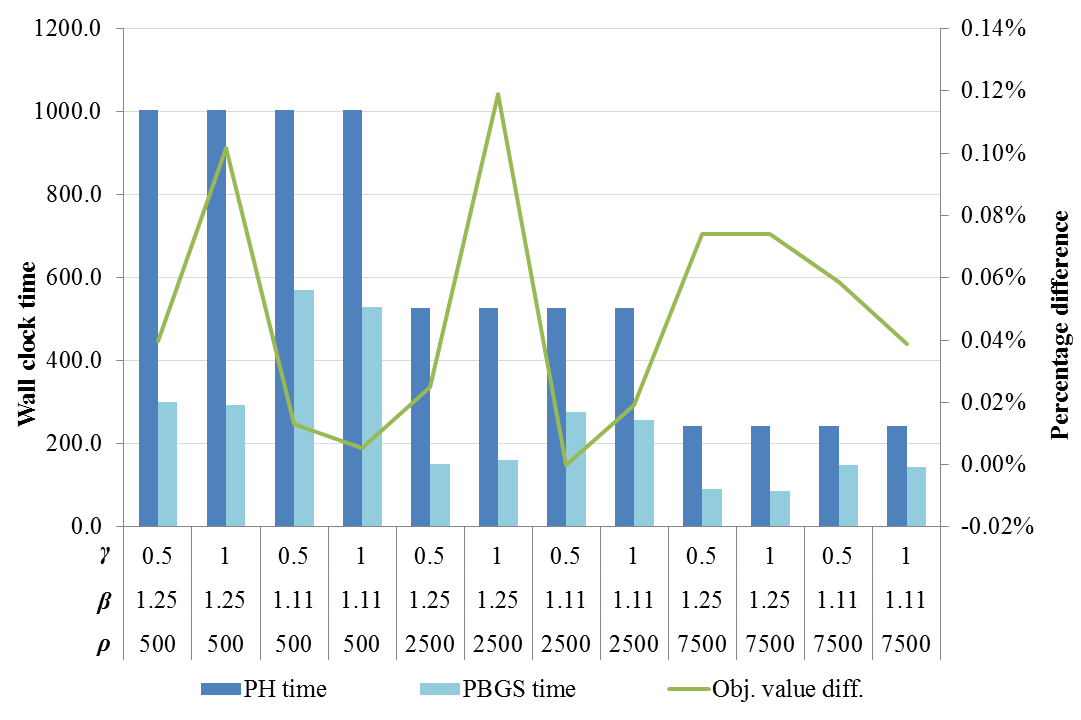}
        \caption{CAP111} 
    \end{subfigure}
    \caption{Results for CAP Problems}
    \label{Fig:CAP}
\end{figure}

\begin{figure}[H]
    \begin{subfigure}[b]{0.5\textwidth}
        \includegraphics[width = \textwidth]{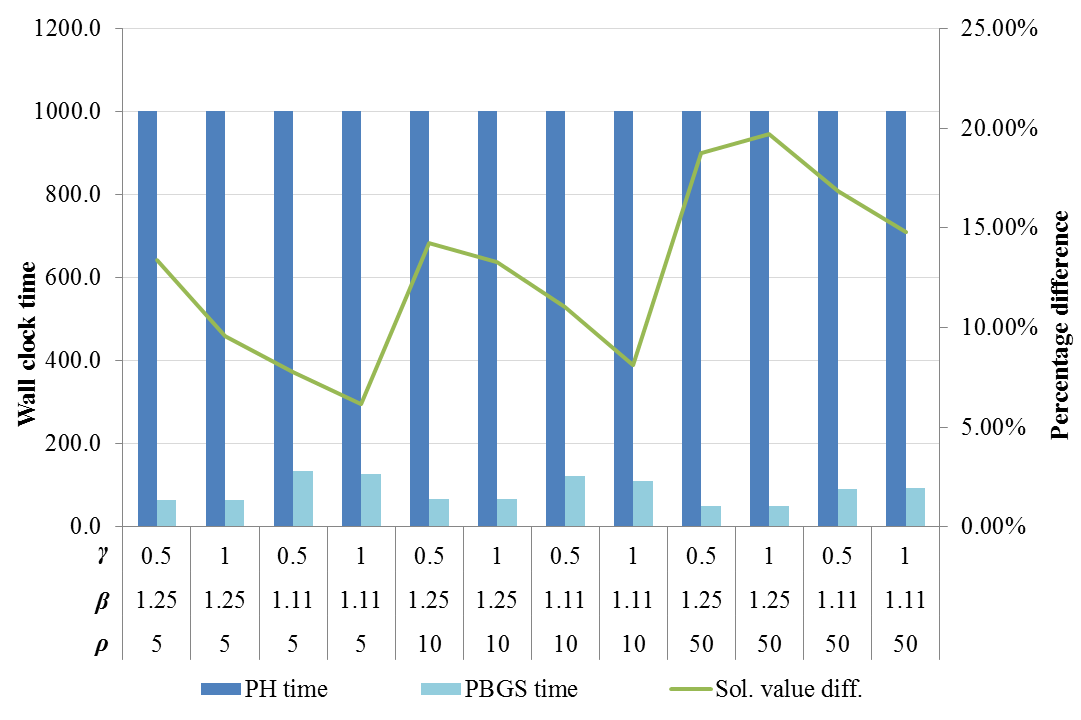}
        \caption{DCAP233} 
    \end{subfigure}
    ~
    \begin{subfigure}[b]{0.5\textwidth}
        \includegraphics[width = \textwidth]{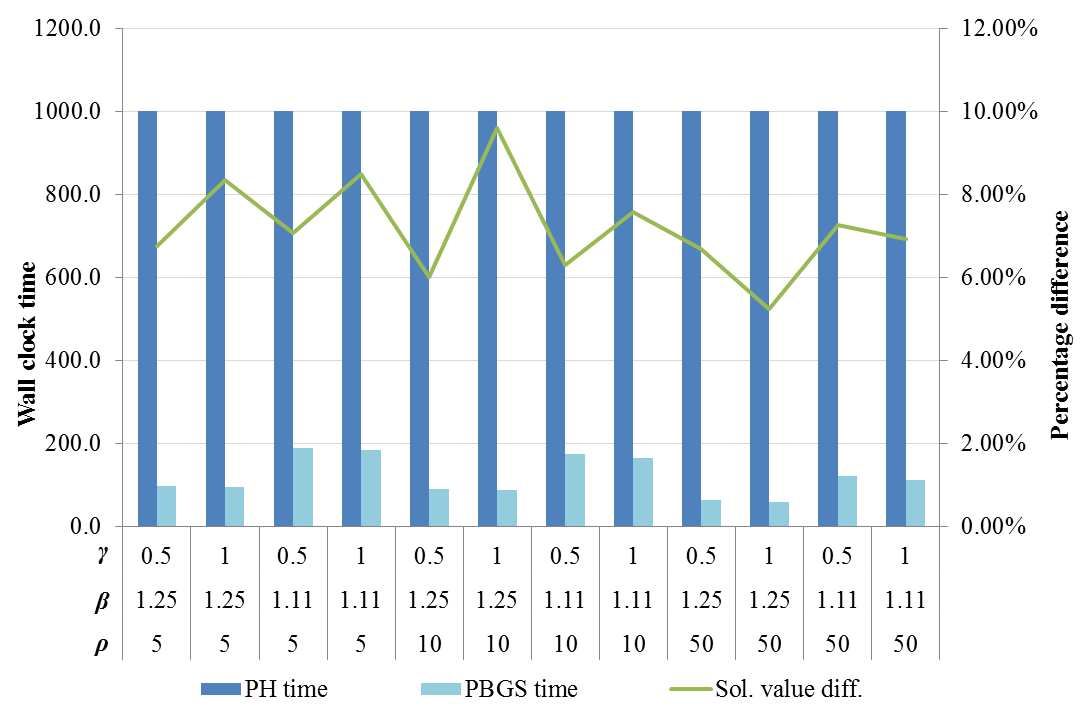}
        \caption{DCAP342} 
    \end{subfigure}
    \caption{Results for DCAP Problems}      
    \label{Fig:DCAP}
\end{figure}

For the CAP instances, all configurations tested with PBGS and PH presented similar values for the objective function, and in most configurations PBGS presented better performance in terms of computational time. For the DCAP instances, in all cases PH terminated due to the time limit of 1000 seconds. For these problems, a comparison in terms of objective function shows that the differences between the objective function value of the solutions found by PGBS and PH are more pronounced. A similar behaviour can be observed in the SSLP instances, in which PBGS outperforms PH in terms of solution times in most cases while providing solutions that are, in the worst case, 0.5\% worse for SSLP5-50 and 5\% worse for SSLP10-50. In the Appendix we present a detailed summary of the statistics for each of the problems, including the fraction of the runs in which PH was not able to converge within the specified time limit. Overall, PBGS seems to be able to obtain comparably good solutions however presenting more reliable convergence behaviour.  

\begin{figure}[H]
	\begin{subfigure}[b]{0.5\textwidth}
		\includegraphics[width = \textwidth]{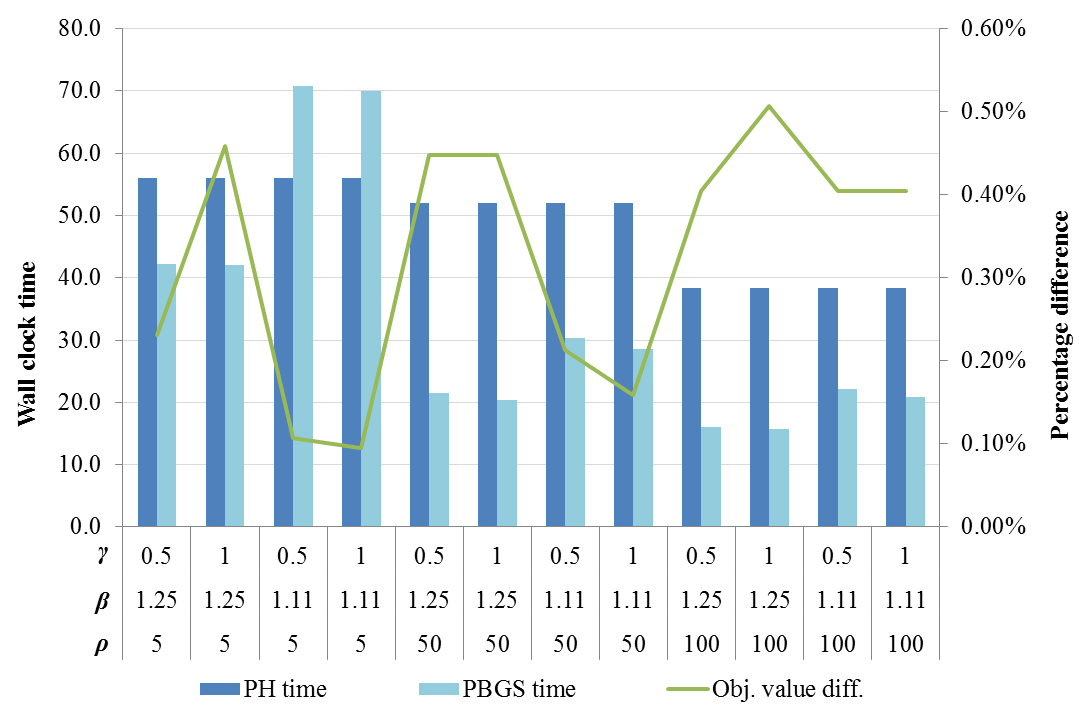}
		\caption{SSLP5-10} 
	\end{subfigure}
	~
	\begin{subfigure}[b]{0.5\textwidth}
		\includegraphics[width = \textwidth]{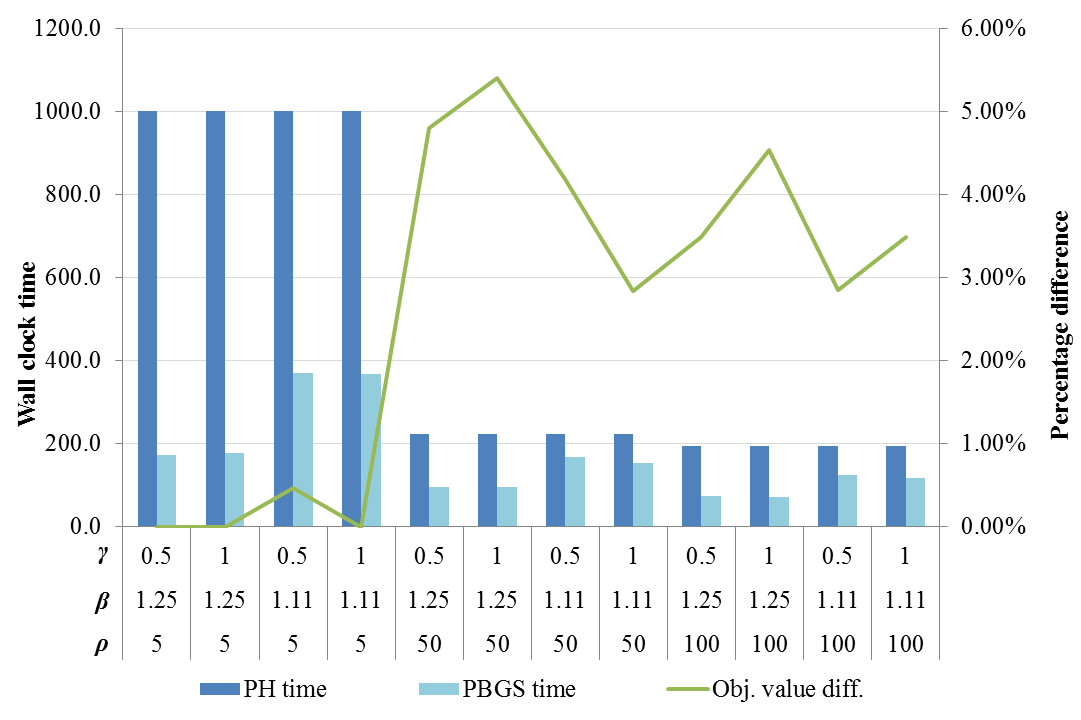}
		\caption{SSLP10-50} 
	\end{subfigure}
	\caption{Results for SSLP Problems}
	\label{Fig:SSLP}
\end{figure}

\section{Conclusions} \label{sec6}

In this paper we have presented an alternative approach for solving stochastic mixed-integer problems based on the combination of penalty-based and block Gauss-Seidel methods. The motivation of such arises from recents theoretical results that motivates the consideration of Lagrangian-based methods under alternative perspectives to approach such problems. 

The computational experiments performed suggest that there is potential for exploiting this framework as it allows the development of a competitive approach in terms of computational efficient. It is worth highlighting that the methodology developed is readily amenable to parallelisation, which is a key point for dealing with large-scale SMIPs.

Further developments of this research could be classified under two distinct standpoints. Under a theoretical perspective, suitable alternative extensions of the block Gauss-Seidel approach into non-smooth non-separable problems are worth investigation. 
A better understanding of how to fine-tune the updates of the penalty coefficients would improve the likelihood (or perhaps even guarantee!) that the block Gauss-Seidel iterations do not display suboptimal stationarity. This would improve the trend of the objective values computed by the main algorithm.
In terms of practical considerations, it would be of interest to evaluate the performance of the proposed approach in contexts other than SMIPs and considering its extension to the multi-stage case.

\section*{Acknowledgements} \label{sec6}
The authors would like to acknowledge the support provided by the Australian Research Council (ARC) grant ARC DP140100985.


\section*{References}

\bibliographystyle{siam}
\bibliography{PaperTwo_Bibliography}
 
\pagebreak 
 
\appendix
\section{Additional computational results} \label{app}

In this Appendix, we present a detailed summary of the computational results obtained. In Tables \ref{tab:App1} to \ref{tab:App6}, row ``Obj. diff." presents the average value (``Average") and standard deviation (``St. dev.") for the relative difference of the objective function value for the solutions obtained with PBGS and PH (with feasibility restored by rounding whenever PH terminated due to the time limit of 1000s). Row ``Speed-up" calculates the relative speed-up that PBGS presented in comparison to PH in terms of wall clock time (values greater than 1 mean that PGBS was faster). Finally, row ``PG conv. fraction" displays the fraction of instances in which PH converged before reaching the specified time limit.   

\vskip 0.25in

 \begin{table}[htbp]
   \centering
   \tiny
   \begin{tabular}{crrrrrrrrrrrrr}
     \toprule
     \multirow{3}[6]{*}{} & $\rho$ & \multicolumn{4}{c}{500} & \multicolumn{4}{c}{2500} & \multicolumn{4}{c}{7500} \\
     
        & $\beta$ & \multicolumn{2}{c}{1.25} & \multicolumn{2}{c}{1.11} & \multicolumn{2}{c}{1.25} & \multicolumn{2}{c}{1.11} & \multicolumn{2}{c}{1.25} & \multicolumn{2}{c}{1.11} \\
        & $\gamma$ & 0.5 & 1  & 0.5 & 1  & 0.5 & 1  & 0.5 & 1  & 0.5 & 1  & 0.5 & 1 \\ \midrule
     \multirow{2}[4]{*}{Obj. diff.} & Average & 0.12\% & 0.12\% & 0.13\% & 0.09\% & 0.06\% & 0.12\% & 0.12\% & 0.20\% & 0.16\% & 0.08\% & 0.17\% & 0.04\% \\ 
        &  St. dev. & 0.13\% & 0.15\% & 0.18\% & 0.14\% & 0.10\% & 0.14\% & 0.14\% & 0.20\% & 0.14\% & 0.10\% & 0.14\% & 0.07\% \\
     \multirow{2}[4]{*}{Speed-up} & Average & 2.02 & 1.12 & 2.14 & 1.24 & 1.63 & 0.93 & 1.57 & 1.01 & 2.12 & 1.44 & 2.28 & 1.60 \\
        & St. dev. & 0.47 & 0.29 & 0.62 & 0.31 & 2.31 & 1.04 & 1.91 & 1.12 & 0.82 & 0.55 & 0.86 & 0.53 \\ \midrule
     \multicolumn{2}{c}{PH conv. Fraction.} & \multicolumn{4}{c}{96.0\%} & \multicolumn{4}{c}{92.0\%} & \multicolumn{4}{c}{94.0\%} \\
     \bottomrule
     \end{tabular}   
   \caption{CAP101}
   \label{tab:App1}%
 \end{table}%
 \begin{table}[htbp]
   \centering
   \tiny 
     \begin{tabular}{ccrrrrrrrrrrrr}
     \toprule
     \multirow{3}[6]{*}{} & $\rho$ & \multicolumn{4}{c}{500} & \multicolumn{4}{c}{2500} & \multicolumn{4}{c}{7500} \\
     
        & $\beta$ & \multicolumn{2}{c}{1.25} & \multicolumn{2}{c}{1.11} & \multicolumn{2}{c}{1.25} & \multicolumn{2}{c}{1.11} & \multicolumn{2}{c}{1.25} & \multicolumn{2}{c}{1.11} \\
        & $\gamma$ & 0.5 & 1  & 0.5 & 1  & 0.5 & 1  & 0.5 & 1  & 0.5 & 1  & 0.5 & 1 \\ \midrule
     \multirow{2}[4]{*}{Obj. diff.} & Average & 0.04\% & 0.10\% & 0.01\% & 0.01\% & 0.02\% & 0.12\% & 0.00\% & 0.02\% & 0.07\% & 0.07\% & 0.06\% & 0.04\% \\
        & St. dev. & 0.06\% & 0.11\% & 0.03\% & 0.01\% & 0.16\% & 0.31\% & 0.17\% & 0.16\% & 0.05\% & 0.05\% & 0.06\% & 0.05\% \\
     \multirow{2}[4]{*}{Speed-up} & Average & 3.95 & 3.67 & 2.06 & 1.87 & 2.98 & 2.93 & 1.60 & 1.73 & 1.97 & 2.06 & 1.19 & 1.25 \\
        & St. dev. & 0.27 & 0.01 & 0.11 & 0.37 & 1.28 & 1.31 & 0.51 & 0.61 & 0.76 & 0.81 & 0.48 & 0.51 \\ \midrule
     \multicolumn{2}{c}{PH conv. Fraction.} & \multicolumn{4}{c}{4.0\%} & \multicolumn{4}{c}{86.0\%} & \multicolumn{4}{c}{92.0\%} \\
     \bottomrule
     \end{tabular}%
   \label{tab:App2}%
   \caption{CAP111}      
\end{table}%
 \begin{table}[htbp]
   \centering
   \tiny
     \begin{tabular}{ccrrrrrrrrrrrr}
     \toprule
     \multirow{3}[6]{*}{} & $\rho$ & \multicolumn{4}{c}{5} & \multicolumn{4}{c}{10} & \multicolumn{4}{c}{50} \\
     
        & $\beta$ & \multicolumn{2}{c}{1.25} & \multicolumn{2}{c}{1.11} & \multicolumn{2}{c}{1.25} & \multicolumn{2}{c}{1.11} & \multicolumn{2}{c}{1.25} & \multicolumn{2}{c}{1.11} \\
        & $\gamma$ & 0.5 & 1  & 0.5 & 1  & 0.5 & 1  & 0.5 & 1  & 0.5 & 1  & 0.5 & 1 \\ \midrule
     \multirow{2}[4]{*}{Obj. diff.} & Average & 13.41\% & 9.56\% & 7.78\% & 6.16\% & 14.22\% & 13.28\% & 11.04\% & 8.10\% & 18.75\% & 19.73\% & 16.86\% & 14.80\% \\
        & St. dev. & 2.18\% & 2.44\% & 2.26\% & 2.53\% & 2.35\% & 2.88\% & 2.78\% & 2.65\% & 2.06\% & 2.33\% & 2.21\% & 2.57\% \\
     \multirow{2}[4]{*}{Speed-up} & Average & N/A & N/A & N/A & N/A & N/A & N/A & N/A & N/A & N/A & N/A & N/A & N/A \\
        & St. dev. & N/A & N/A & N/A & N/A & N/A & N/A & N/A & N/A & N/A & N/A & N/A & N/A \\
     \multicolumn{2}{c}{PH conv. Fraction.} & \multicolumn{4}{c}{0.0\%} & \multicolumn{4}{c}{0.0\%} & \multicolumn{4}{c}{0.0\%} \\
     \bottomrule
     \end{tabular}%
   \label{tab:App3}%
   \caption{DCAP233}         
 \end{table}%
 \begin{table}[htbp]
   \centering
   \tiny
     \begin{tabular}{ccrrrrrrrrrrrr}
     \toprule
     \multirow{3}[6]{*}{} & $\rho$ & \multicolumn{4}{c}{5} & \multicolumn{4}{c}{10} & \multicolumn{4}{c}{50} \\
     
        & $\beta$ & \multicolumn{2}{c}{1.25} & \multicolumn{2}{c}{1.11} & \multicolumn{2}{c}{1.25} & \multicolumn{2}{c}{1.11} & \multicolumn{2}{c}{1.25} & \multicolumn{2}{c}{1.11} \\
        & $\gamma$ & 0.5 & 1  & 0.5 & 1  & 0.5 & 1  & 0.5 & 1  & 0.5 & 1  & 0.5 & 1 \\  \midrule
     \multirow{2}[4]{*}{Obj. diff.} & Average & 6.77\% & 8.36\% & 7.08\% & 8.51\% & 6.01\% & 9.61\% & 6.31\% & 7.59\% & 6.69\% & 5.24\% & 7.27\% & 6.92\% \\
        & St. dev. & 4.78\% & 4.45\% & 4.47\% & 4.39\% & 3.26\% & 5.36\% & 5.07\% & 3.70\% & 3.82\% & 2.59\% & 2.90\% & 3.67\% \\
     \multirow{2}[4]{*}{Speed-up} & Average & \multicolumn{1}{c}{N/A} & N/A & N/A & N/A & N/A & N/A & N/A & N/A & N/A & N/A & N/A & N/A \\
        & St. dev. & N/A & N/A & N/A & N/A & N/A & N/A & N/A & N/A & N/A & N/A & N/A & N/A \\
     \multicolumn{2}{c}{PH conv. Fraction.} & \multicolumn{4}{c}{0.0\%} & \multicolumn{4}{c}{0.0\%} & \multicolumn{4}{c}{0.0\%} \\
     \bottomrule
     \end{tabular}%
   \label{tab:App4}%
   \caption{DCAP342}
  \end{table}%
 \begin{table}[htbp]
   \centering
   \tiny
     \begin{tabular}{ccrrrrrrrrrrrr}
     \toprule
\multirow{3}[6]{*}{} & $\rho$ & \multicolumn{4}{c}{5} & \multicolumn{4}{c}{50} & \multicolumn{4}{c}{100} \\
     
        & $\beta$ & \multicolumn{2}{c}{1.25} & \multicolumn{2}{c}{1.11} & \multicolumn{2}{c}{1.25} & \multicolumn{2}{c}{1.11} & \multicolumn{2}{c}{1.25} & \multicolumn{2}{c}{1.11} \\
        & $\gamma$ & 0.5 & 1  & 0.5 & 1  & 0.5 & 1  & 0.5 & 1  & 0.5 & 1  & 0.5 & 1 \\   \midrule
     \multirow{2}[4]{*}{Obj. diff.} & Average & 0.23\% & 0.46\% & 0.11\% & 0.10\% & 0.45\% & 0.45\% & 0.21\% & 0.16\% & 0.40\% & 0.51\% & 0.40\% & 0.40\% \\
        & St. dev. & 0.76\% & 1.14\% & 0.49\% & 0.32\% & 1.08\% & 1.08\% & 0.64\% & 0.53\% & 1.07\% & 1.25\% & 1.07\% & 1.07\% \\
     \multirow{2}[4]{*}{Speed-up} & Average & 1.29 & 1.29 & 0.76 & 0.81 & 1.32 & 1.40 & 0.93 & 1.03 & 1.12 & 1.21 & 0.82 & 0.86 \\
        & St. dev. & 0.60 & 0.56 & 0.30 & 0.31 & 0.91 & 0.96 & 0.61 & 0.67 & 0.68 & 0.82 & 0.58 & 0.49 \\
     \multicolumn{2}{c}{PH conv. Fraction.} & \multicolumn{4}{c}{100.0\%} & \multicolumn{4}{c}{98.0\%} & \multicolumn{4}{c}{98.0\%} \\
     \bottomrule
     \end{tabular}%
     \caption{SSLP5-50}
   \label{tab:App5}%
 \end{table}%
     
 \begin{table}[htbp]
   \centering
   \tiny
     \begin{tabular}{ccrrrrrrrrrrrr}
     \toprule
\multirow{3}[6]{*}{} & $\rho$ & \multicolumn{4}{c}{5} & \multicolumn{4}{c}{50} & \multicolumn{4}{c}{100} \\
     
        & $\beta$ & \multicolumn{2}{c}{1.25} & \multicolumn{2}{c}{1.11} & \multicolumn{2}{c}{1.25} & \multicolumn{2}{c}{1.11} & \multicolumn{2}{c}{1.25} & \multicolumn{2}{c}{1.11} \\
        & $\gamma$ & 0.5 & 1  & 0.5 & 1  & 0.5 & 1  & 0.5 & 1  & 0.5 & 1  & 0.5 & 1 \\    \midrule
     \multicolumn{2}{c}{} & 0.5 & 1  & 0.5 & 1  & 0.5 & 1  & 0.5 & 1  & 0.5 & 1  & 0.5 & 1 \\
     \multirow{2}[4]{*}{Obj. diff.} & Average & 0.00\% & 0.00\% & 0.47\% & 0.00\% & 4.80\% & 5.41\% & 4.19\% & 2.84\% & 3.50\% & 4.54\% & 2.85\% & 3.50\% \\
        & St. dev. & 0.0\% & 0.0\% & 0.5\% & 0.0\% & 12.5\% & 12.9\% & 12.0\% & 10.5\% & 8.8\% & 9.9\% & 8.0\% & 8.8\% \\
     \multirow{2}[4]{*}{Speed-up} & Average & 7.95 & 7.86 & 2.95 & 4.07 & 2.27 & 2.35 & 1.35 & 1.41 & 1.69 & 1.72 & 1.03 & 1.10 \\
        & St. dev. & 2.14 & 2.08 & 0.36 & 1.13 & 1.11 & 1.29 & 0.78 & 0.74 & 1.04 & 1.05 & 0.69 & 0.75 \\
     \multicolumn{2}{c}{PH conv. Fraction.} & \multicolumn{4}{c}{4.0\%} & \multicolumn{4}{c}{96.0\%} & \multicolumn{4}{c}{90.0\%} \\
     \bottomrule
     \end{tabular}%
   \caption{SSLP10-50}          
   \label{tab:App6}%
\end{table}%

\end{document}